\documentclass[a4paper]{amsart}

\usepackage{geometry}          
\usepackage{amsmath, amstext, ,amscd, amssymb, mathrsfs, amsthm}
\usepackage{stmaryrd}
\usepackage{mathrsfs}
\usepackage{hyperref}
\usepackage{enumitem}
\usepackage[all]{xy}

\newcommand{\R}{\mathbb{R}}
\newcommand{\C}{\mathbb{C}}

\newcommand{\Z}{\mathbb{Z}}
\newcommand{\N}{\mathbb{N}}

\newcommand{\F}{\mathrm{F}}
\newcommand{\G}{\mathrm{G}}

\newcommand{\smooth}{\mathscr{C}^{\infty}}
\newcommand{\D}{\mathcal{D}}
\newcommand{\sob}{\boldsymbol{H}}

\newcommand{\ic}{\sqrt{-1}}
\newcommand{\lie}{\mathcal{L} }
\newcommand{\bw}{\overline{w}}

\newcommand{\derpar}[2]{\frac{\partial#1}{\partial#2}}

\newcommand{\e}{\varepsilon}

\newcommand{\E}{\mathbb{E}}

\newcommand{\FF}{\mathscr{F}}
\newcommand{\Ad}{\mathrm{Ad}}

\newcommand{\db}{\bar{\partial}}
\newcommand{\Wedge}{\Lambda^{0,\bullet}}

\newcommand{\End}{\mathrm{End}}
\newcommand{\Sp}{\mathrm{Sp}}

\newcommand{\LL}{\mathscr{L}}

\newcommand{\kahler}{\text{K\"{a}hler }}
\newcommand{\n}{\nabla}
\newcommand{\ol}{\overline}

\newcommand{\limarrow}[2]{\xrightarrow[#1\to #2]{}}
\newcommand{\isom}{\overset{\sim}{\longrightarrow}}

\newcommand{\g}{\mathfrak{g}}
\newcommand{\wt}[1]{\widetilde{#1}}
\newcommand{\boldrm}[1]{{\boldsymbol{\mathrm{#1}}}}
\newcommand{\wh}[1]{\widehat{#1}}
\newcommand{\V}{\mathcal{V}_{\gamma}}

\renewcommand{\Re}{\mathrm{Re}}
\renewcommand{\Im}{\mathrm{Im}}

\DeclareMathOperator{\Hom}{Hom}

\DeclareMathOperator{\rank}{rk}

\DeclareMathOperator{\Id}{Id} 
\DeclareMathOperator{\supp}{supp}
\DeclareMathOperator{\tr}{Tr}

\newtheorem{prop}{Proposition}[section]
\newtheorem{thm}[prop]{Theorem}

\newtheorem{lemme}[prop]{Lemma}

\newtheorem{cor}[prop]{Corollary}
\newtheorem{assumption}[prop]{Assumption}

\theoremstyle{definition}

\newtheorem{defn}[prop]{Definition}

\theoremstyle{remark}

\newtheorem{rem}[prop]{Remark}

\numberwithin{equation}{section}
\begin{document}

\title{$G$-invariant Holomorphic Morse inequalities}
\date{\today}

\author{Martin {\sc{Puchol}}} 
\address{Universit\'{e} Paris Diderot--Paris 7, Campus des Grands Moulins, B\^{a}timent Sophie Germain, case 7012, 75205 Paris Cedex 13}
\email{martin.puchol@imj-prg.fr}

\begin{abstract}
Consider an action of a connected compact Lie group on a compact complex manifold $M$, and two equivariant vector bundles $L$ and $E$ on $M$, with $L$ of rank 1. The purpose of this paper is to establish holomorphic Morse inequalities \`{a} la Demailly for the invariant part of the Dolbeault cohomology of tensor powers of $L$ twisted by $E$. To do so, we define a moment map $\mu$ by the Kostant formula and we define the reduction of $M$ under a natural hypothesis on $\mu^{-1}(0)$. Our inequalities are given in term of the curvature of the bundle induced by $L$ on this reduction.
\end{abstract}

\maketitle    
\tableofcontents


\section{Introduction}

Morse Theory investigates the topological information carried by Morse functions on a manifold and in particular their critical points. Let $f$ be a Morse function on a compact manifold of real dimension $n$. We suppose that $f$ has isolated critical points. Let $m_j$, $(0 \leq j \leq n)$ be the the number of critical points of $f$ of Morse index  $j$, and let $b_j$ be the Betti numbers of the manifold. Then the strong Morse inequalities states that for $0\leq q \leq n$,
\begin{equation}
\label{MI-reelles}
\sum_{j=0}^q(-1)^{q-j} b_j \leq \sum_{j=0}^q(-1)^{q-j} m_j,
\end{equation}
with equality if $q=n$. From \eqref{MI-reelles}, we get the weak Morse inequalities:
\begin{equation}
b_j\leq m_j \qquad \text{for} \quad 0\leq j\leq n.
\end{equation}

In his seminal paper \cite{MR683171}, Witten  gave an analytic proof of the Morse inequalities by analyzing the spectrum of the Schr\"{o}dinger operator $\Delta_t=\Delta +t^2|df|^2+tV$, where $t>0$ is a real parameter and $V$ an operator of order 0. For $t\to +\infty$, Witten shows that the spectrum of $\Delta_t$ approaches in some sense the spectrum of a sum of harmonic oscillators attached to the critical point of $f$.  

In \cite{demailly}, Demailly established analogous asymptotic Morse inequalities for the Dolbeault cohomology associated with high tensor powers $L^p:=L^{\otimes p}$ of a holomorphic Hermitian line bundle $(L,h^L)$ over a compact complex manifold $(M,J)$.  The inequalities of Demailly give asymptotic bounds on the Morse sums of the Betti numbers of $\db$ on $L^p$ in terms of certain integrals of the Chern curvature $R^L$ of $(L,h^L)$. More precisely,  we define $\dot{R}^L\in \End(T^{(1,0)}M)$ by $g^{TM}(\dot{R}^Lu,\ol{v}) =R^L(u,\ol{v})$ for $u,v\in T^{(1,0)}M$, where $g^{TM}$ is a $J$-invariant Riemannian metric on $TM$. We denote by $M(\leq q)$ the set of points where $\dot{R}^L$ is non-degenerate and have at most $q$ negative eigenvalues, and we set $n=\dim_\C M$. Then we have for $0\leq q \leq n$
\begin{equation}
\label{IM}
\sum_{j=0}^q (-1)^{q-j} \dim H^j(M, L^p) \leq \frac{p^n}{n!} \int_{M(\leq q)} (-1)^q \left( \frac{\ic}{2\pi}R^{L}\right)^{n} +o(p^n),
\end{equation}
with equality if $q=n$. Here $H^j(M,L^p)$ denotes the Dolbeault cohomology in bidegree $(0,j)$, which is also the $j$-th group of cohomology of the sheaf of holomorphic sections of $L^p$. 

These inequalities have found numerous applications. In particular, Demailly used them in \cite{demailly} to find new geometric characterizations of Moishezon spaces, which improve Siu's solution in \cite{MR755233,MR797421} of the Grauert-Riemenschneider conjecture \cite{MR0302938}. Another notable application of the holomorphic Morse inequalities is the proof of the effective Matsusaka theorem by Siu \cite{MR1275204,MR1603616}. Recently, Demailly used these inequalities in \cite{MR2918158} to prove a significant step of a generalized version of the Green-Griffiths-Lang conjecture.

To prove these inequalities, the key remark of Demailly was that in the formula for the Kodaira Laplacian $\square_p$ associated with $L^p$, the metric of $L$ plays formally the role of the Morse function in the paper Witten \cite{MR683171}, and that the parameter $p$ plays the role of the parameter $t$. Then the Hessian of the Morse function becomes the curvature of the bundle. The proof of Demailly was based on the study of the semi-classical behavior as $p\to +\infty$ of the spectral counting functions of $\square_p$. Subsequently, Bismut gave an other proof of  the holomorphic Morse inequalities in \cite{MR886814} by adapting his heat kernel proof of the Morse inequality \cite{MR852155}. The key point is that we can compare the left hand side of \eqref{IM} with the alternate trace of the heat kernel acting on forms of degree $\leq q$, i.e.,
\begin{equation}
\label{Morse-sum-and-supertrace}
\sum_{j=0}^q (-1)^{q-j} \dim H^j(M, L^p) \leq \sum_{j=0}^q (-1)^{q-j} \tr^{\Omega^{0,j}(M,L^p)}\left[ \exp\Big(-\frac{u}{p}\square_p\Big)\right],
\end{equation}
with equality if $q=n$. Then, Bismut obtained the holomorphic Morse inequalities by showing the convergence of the heat kernel thanks to probability theory. Demailly \cite {MR1128538} and Bouche \cite{MR1056777} gave an analytic approach of this result. In \cite{ma-marinescu}, Ma and Marinescu  gave a new proof of this convergence, replacing the probabilistic arguments of Bismut  \cite{MR886814} by arguments inspired by the analytic localization techniques of Bismut-Lebeau \cite[Chap. 11]{bismut-lebeau}.

When the bundle $L$ is positive, \eqref{IM} is a consequence of the Hirzebruch-Riemann-Roch theorem and of the Kodaira vanishing theorem, and  reduces to
\begin{equation}
\dim H^0(M, L^p)=\frac{p^n}{n!} \int_{M}  \big( \frac{\ic}{2\pi}R^{L}\big)^{n} +o(p^n).
\end{equation}
 In this case, a local estimate can be obtained by the study of the asymptotic of the Bergman kernel (the kernel of the orthogonal projection from $\smooth(M,L^p)$ onto $H^0(M,L^p)$) when $p\to+\infty$.  We refer to \cite{ma-marinescu} and the reference therein for the study of the Bergman kernel. 

In the equivariant case, a connected compact Lie group $G$ acts on  $M$ and its action lifts on $L$. When $L$ is positive, Ma and Zhang \cite{ma-zhang} have studied the invariant Bergman kernel, i.e., the kernel of the projection from $\smooth(M,L^p)$ onto the $G$-invariant part of $H^0(M,L^p)$. Let $\mu$ be the moment map associated with the $G$-action on $M$ (see \eqref{def-mu}). Ma and Zhang \cite{ma-zhang} established that the invariant Bergman kernel concentrate to any neighborhood $U$ of $\mu^{-1}(0)$, and that near $\mu^{-1}(0)$, we have a full off-diagonal asymptotic development. They also obtain a fast decay of the invariant Bergman kernel in the normal directions to $\mu^{-1}(0)$, which does not appear in the classical case. 

In this paper, we establish $G$-invariant holomorphic Morse inequalities  under certain natural condition, in the context of Ma-Zhang \cite{ma-zhang} but without the assumption that $L$ is  positive.

More precisely, we consider an action of a connected compact Lie group $G$ on a compact complex manifold $M$ and two $G$-equivariant vector bundles $L$ and $E$ on $M$, with $L$ of rank 1, and we  establish asymptotic holomorphic Morse inequalities similar to \eqref{IM} for the $G$-invariant part of the Dolbeault cohomology of $ L^p\otimes E$ (see Theorems \ref{IMI} and \ref{IMI-general}). To do so, we define a ``moment map" $\mu\colon M\to \mathrm{Lie}(G)$ by the Kostant formula and we define the reduction of $M$ under natural hypothesis on $\mu^{-1}(0)$ (see Assumption \ref{assum-regular-value}). Our inequalities are then given in term of the curvature of the bundle induced by $L$ on this reduction, and the integral in \eqref{IM} will be over subsets of the reduction.

 A new feature in our setting when compared to Demailly's result is the localization near $\mu^{-1}(0)$. We use a heat kernel method inspired by \cite{MR886814} (see also \cite[Sect. 1.6-1.7]{ma-marinescu}), the key being that an analogue of \eqref{Morse-sum-and-supertrace} still holds (see Theorem \ref{Euler-et-chaleur}) for the Kodaira Laplacian restricted to the space of invariant forms. We show that the heat kernel will concentrate in any neighborhood $U$ of $\mu^{-1}(0)$, and we study the asymptotic of the heat kernel near $\mu^{-1}(0)$. For this last part, we work with the operator  induced  by the Kodaira Laplacian on the quotient of $U$. However, as we will have to integrate the heat kernel in the normal directions to $\mu^{-1}(0)$, we need a more precise convergence result that in \cite[Sect. 1.6]{ma-marinescu}. Indeed we also need to prove a uniform fast decay of the heat kernel in the normal directions, which is analogous to the decay  encountered in \cite[Thm. 0.2]{ma-zhang}. Our approach is largely inspired by \cite{ma-zhang}.
 
 Note that in the literature, there exists another type of equivariant holomorphic Morse inequalities \cite{MR792703,MR1472894,MR1601858}, which relate the Dolbeault cohomology groups of the fixed point-set of  a compact K\"{a}hler manifold $M$ endowed with an action of a compact connected Lie group $G$ to the Dolbeault cohomology groups of $M$ itself.

We now give more details about our results. Let $(M,J)$ be a connected compact complex manifold. Let $n=\dim_\C M$. Let $(L,h^L)$ be a holomorphic Hermitian line bundle on $M$, and $(E,h^E)$ a holomorphic Hermitian vector bundle on $M$. We denote the Chern (i.e., holomorphic and Hermitian) connections of $L$ and $E$ respectively by $\nabla^L$ and $\nabla^E$, and their respective curvatures by $R^L=( \nabla^L)^2$ and $R^E=( \nabla^E )^2$. Let $\omega$ be the first Chern form of $(L,h^L)$, i.e., the $(1,1)$-form defined by
\begin{equation}
\label{prequantization}
\omega = \frac{\ic}{2\pi} R^L.
\end{equation}
We \emph{do not assume} that $\omega$ is a positive $(1,1)$-form.

Let  $G$ be a connected compact Lie group with Lie algebra $\g$. Let $d=\dim_\R G$. We assume that $G$ acts holomorphically on $(M,J)$, and that the action lifts to a holomorphic action on $L$ and $E$. We assume that $h^L$ and $h^E$ are preserved by the $G$-action. Then $R^L$, $R^E$ and $\omega$ are $G$-invariant forms.

In the sequel, if $F$ is any $G$-representation, we denote by $F^G$ the space of elements of $F$ invariant under the action of $G$. The infinitesimal action of $K\in \g$ on any $F$ will be denoted by $\lie_K^F$, or simply by $\lie_K$ when it entails no confusion.

For  $K\in \g$, let $K^M$ be the vector field on $M$ induced by $K$ (see \eqref{def-KM}). We can define a map $\mu\colon M\to \g^*$  by the Kostant formula
\begin{equation}
\label{def-mu}
\mu(K) = \frac{1}{2i\pi} \left( \n^L_{K^M} - \lie_K \right).
\end{equation}
Then for any $K\in \g$ (see Lemma \ref{mu=moment-thm}),
 \begin{equation}
d\mu (K) = i_{K^M} \omega.
\end{equation}
 Moreover the set defined by
\begin{equation}
\label{def-P}
P=\mu^{-1}(0)
\end{equation}
 is stable by $G$.

We make the following assumption:
\begin{assumption}
\label{assum-regular-value}
$0$ is a regular value of $\mu$.
\end{assumption}

Under Assumption \ref{assum-regular-value}, $P$ is a submanifold. Moreover, by Lemma \ref{action-loc-libre}, $G$ acts locally freely  on $P$, so that the quotient $M_G=P/G$  is an orbifold, which we call the \emph{reduction} of $M$. For definition and basic properties of orbifolds, we refer to \cite[Sect. 5.4]{ma-marinescu} for instance. The projection $P\to M_G$ is denoted by $\pi$.

We denote by $TY$ the tangent bundle of the $G$-orbits in $P$. As $G$ acts locally freely on $P$, we know that $TY=\mathrm{Span}(K^M,K\in \g)$ and that it is a vector bundle on $P$.

The following analogue of the classical \kahler reduction (see \cite{guillemin-sternberg}) holds. 
\begin{thm}
\label{descent}
The complex structure $J$ on $M$ induces a complex structure $J_G$ on $M_G$, for which the orbifold bundles $L_G,E_G$  induced by $L,E$ on $M_G$ are holomorphic. Moreover, the form $\omega$ descends to a form $\omega_G$ on $M_G$ and if $R^{L_G}$ is the Chern curvature of $L_G$ for the metric $h^{L_G}$ induced by $h^L$, then 
\begin{equation}
\omega_G = \frac{\ic}{2\pi} R^{L_G}.
\end{equation}
 Finally,  for $x\in M_G$, $\pi_*$ induces an isomorphism
\begin{equation}
(\ker \omega_G)_x \simeq (\ker \omega)|_{\pi^{-1}(x)}.
\end{equation}
\end{thm}

Let $b^L$ be the bilinear form on $TM$
\begin{equation}
\label{def-b^L}
b^L(\cdot \, , \cdot) = \frac{\ic}{2\pi} R^L(\cdot \,, J\cdot)=\omega(\cdot \,, J\cdot).
\end{equation} 
Then we will show in Lemma \ref{bL-ND} that when restricted to $TY\times TY$, the bilinear form $b^L$ is non-degenerate on $P$. In particular, the signature of $b^L|_{TY\times TY}$ is constant on $P$. We denote by $(r,d-r)$ this signature, i.e., in any orthogonal (with respect to $b^L$) basis of $TY|_P$, the matrix of $b^L$ will have $r$ negative diagonal elements and $d-r$ positive diagonal elements.

We define $\dot{R}^{L_G}\in \End(T^{(1,0)}M_G)$ by $g(\dot{R}^{L_G}u,\ol{v}) =R^{L_G}(u,\ol{v})$ for $u,v\in T^{(1,0)}M_G$, where $g$ is a $J_G$-invariant Riemannian metric on the orbifold tangent bundle $TM_G$. We denote by $M_G(q)$ the set of $x\in M_G$ such that $\dot{R}^{L_G}_x$ is invertible and has exactly $q$ negative eigenvalues, with the convention that if $q\notin \{0,\dots,n-d\}$, then $M_G(q)=\emptyset$. Set $M_G(\leq q) = \cup_{i\leq q} M_G(i)$. Note that $M_G(q)$ does not depend on the metric $g$.

As $G$ preserves every structure we are given, it acts naturally on the Dolbeault cohomology $H^\bullet (M,L^p\otimes E)$. The following theorem is the main result of our paper.
\begin{thm}
\label{IMI}
 Assume that $G$ acts effectively on $M$ (i.e., the only element of $G$ acting as $\Id_M$ is the identity). Then as $p\to +\infty$, the following strong Morse inequalities hold for $q\in \{ 1,\dots, n\}$
 \begin{equation}
\sum_{j=0}^q (-1)^{q-j}\dim H^j(M,L^p\otimes E)^G \leq \rank(E) \frac{p^{n-d}}{(n-d)!}  \int_{M_G(\leq q-r)} (-1)^{q-r} \omega_G^{n-d} + o(p^{n-d}),
\end{equation}
with equality for $q=n$.

In particular, we get the weak Morse inequalities
 \begin{equation}
\dim H^q(M,L^p\otimes E)^G \leq \rank(E) \frac{p^{n-d}}{(n-d)!}  \int_{M_G( q-r)} (-1)^{q-r} \omega_G^{n-d} + o(p^{n-d}).
\end{equation}
\end{thm}

\begin{rem}
We assume temporarily that $G$ acts freely on $P$, so that $M_G$ is a manifold.

If $L$ is positive, then $\omega$ is a \kahler form and $\mu$ is a genuine moment map. Moreover, $(M_G,\omega_G)$ is the usual \kahler reduction of $M$ (see \cite{guillemin-sternberg}). Zhang \cite[Theorem 1.1 and Proposition 1.2]{doi:10.1142/S0219199799000122} proved that  in this case quantization and reduction commute: for $p$ large enough,
\begin{equation}
\label{[Q,R]=0}
H^\bullet (M,L^p\otimes E)^G  \simeq H^\bullet( M_G,L_G^p\otimes E_G).
\end{equation}
We refer to Vergne's Bourbaki seminar \cite{vergne} for a survey on the Guillemin-Sternberg geometric quantization conjecture.

In particular, as in the non equivariant setting, Theorem \ref{IMI} is, in this case, a consequence of \eqref{[Q,R]=0} and of the Hirzebruch-Riemann-Roch theorem and of the Kodaira vanishing theorem, both applied  on~$M_G$.

We prove here that even if  $\omega$ is degenerate or if $G$ does not act freely on $P$, under Assumption~\ref{assum-regular-value}, we have the same estimate for $\sum_{j=0}^q (-1)^{q-j}\dim H^j(M,L^p\otimes E)^G$ as the one given by the holomorphic Morse inequalities on $M_G$ for $\sum_{j=0}^{q-r} (-1)^{q-j}\dim H^j(M_G,L_G^p\otimes E_G)$.
\end{rem}

Theorem \ref{IMI} is in fact a particular case of the more general Theorem \ref{IMI-general} below.

Set 
\begin{equation}
G^0=\{g\in G \: : \:  g\cdot x=x \text{ for any }x\in M\},
\end{equation}
which is a finite normal subgroup of $G$. Note that we will see in \eqref{G0(M)=G0(P)} that we also have $G^0=\{g\in G \: : \:  g\cdot x=x \text{ for any }x\in P\}$.

Observe that $\dim(L^p_v\otimes E_v)^{G^0}$ does not depend on $v\in M$. We will thus denote it simply by $\dim(L^p\otimes E)^{G^0}$.

\begin{thm}
 \label{IMI-general}
As $p\to +\infty$, the following strong Morse inequalities hold for $q\in \{ 1,\dots, n\}$
 \begin{multline}
\sum_{j=0}^q (-1)^{q-j}\dim H^j(M,L^p\otimes E)^G  \\
\leq \dim(L^p\otimes E)^{G^0} \frac{p^{n-d}}{(n-d)!}  \int_{M_G(\leq q-r)} (-1)^{q-r} \omega_G^{n-d} + o(p^{n-d}),
\end{multline}
with equality for $q=n$.

In particular, we get the weak Morse inequalities
 \begin{equation}
\dim H^q(M,L^p\otimes E)^G \leq \dim(L^p\otimes E)^{G^0} \frac{p^{n-d}}{(n-d)!}  \int_{M_G( q-r)} (-1)^{q-r} \omega_G^{n-d} + o(p^{n-d}).
\end{equation}
\end{thm}

\begin{rem} 
The integer $\dim(L^p\otimes E)^{G^0}$ depends on $p$. However, as $G^0$ is finite and acts by rotations on $L$, there exists $k\in \N$ (a divisor of the cardinal of $G^0$) such that $G^0$ acts trivially on $L^k$. In particular,  we have $\dim(L^{kp}\otimes E)^{G^0}=\dim E^{G^0}$.
\end{rem}

We now explain what are the main steps of our proof.

Let $g^{TM}$ be a $J$- and $G$-invariant metric on $TM$.  Let $dv_M$ be the corresponding Riemannian volume on $M$, and let $\n^{TM}$ be the Levi-Civita connection on $(TM,g^{TM})$. Let $\db^{L^p\otimes E}$ be the Dolbeault operator acting on $\Omega^{0,\bullet}(M,L^p\otimes E)$. Let $\db^{L^p\otimes E,*}$ be its dual with respect to the $L^2$ product induced by $g^{TM}$, $h^L$ and $h^E$ on $\Omega^{0,\bullet}(M,L^p\otimes E)$. We set
\begin{equation}
\label{def-D_p}
D_p = \sqrt{2} \left( \db^{L^p\otimes E} + \db^{L^p\otimes E,*} \right),
\end{equation}
and we denote by  $e^{-uD_p^2}$ the associated heat kernel.

We denote  $P_G$  the orthogonal projection from $\Omega^{0,\bullet}(M,L^p\otimes E)$ onto $\Omega^{0,\bullet}(M,L^p\otimes E)^G$. Let $\big(P_Ge^{-\frac{u}{p}D_p^2}P_G\big)(v,v')$ be the Schwartz kernel of $P_Ge^{-\frac{u}{p}D_p^2}P_G$ with respect to $dv_M(v')$.

Note that the operator $D_p^2$ acts on $\Omega^{0,\bullet}(M,L^p\otimes E)^G$ (i.e., commutes with $P_G$) and preserves the $\Z$-grading. we denote by $\tr_q[P_Ge^{-\frac{u}{p}D_p^2}P_G]$ the trace of $P_Ge^{-\frac{u}{p}D_p^2}P_G$ acting on $\Omega^{0,q}(M,L^p\otimes E)$. We then have an analogue of \eqref{Morse-sum-and-supertrace}:
\begin{thm}
\label{Euler-et-chaleur}
 For any $u>0$, $p\in \N^*$ and $0\leq q \leq n$, we have
 \begin{equation}
 \label{Euler-et-chaleur-eq}
\sum_{j=0}^q (-1)^{q-j} \dim H^j(M,L^p\otimes E)^G \leq \sum_{j=0}^q (-1)^{q-j} \tr_j\big[P_Ge^{-\frac{u}{p}D_p^2}P_G\big],
\end{equation}
with equality for $q=n$.
\end{thm}

We now give the estimates on $P_Ge^{-\frac{u}{p}D_p^2}P_G$ to treat the right-hand side of \eqref{Euler-et-chaleur-eq}.

Let $U$ be a small open $G$-invariant neighborhood of $P$, such that $G$ acts locally freely on its closure $\ol{U}$. 

First, we have away from $P$ the following theorem. The analogous result of Ma-Zhang for the invariant Bergman kernel is \cite[Thm. 0.1]{ma-zhang}.
\begin{thm}
\label{thm-estimate-away}
 For any fixed $u>0$  and $k,\ell \in \N$, there exists $C>0$ such that for any $p\in \N^*$ and $v,v'\in M$ with $v,v'\in M\setminus U$,
 \begin{equation}
\left|P_Ge^{-\frac{u}{p}D_p^2}P_G(v,v') \right|_{\mathscr{C}^\ell}\leq Cp^{-k},
\end{equation}
where $|\cdot |_{\mathscr{C}^\ell}$ is the $\mathscr{C}^\ell$-norm induced by $\n^L$, $\n^E$, $\n^{TM}$, $h^L$, $h^E$ and $g^{TM}$.
\end{thm}

We now turn to the ``near $P$'' asymptotic of the heat kernel. To explain simply this asymptotic, we assume now that $G$ acts freely on $P$. We can thus also assume that $G$ acts freely on $\ol{U}$. Let $B=U/G$. Then $M_G$ and $B$ are here genuine manifolds. We will explain in Section \ref{Sect-action-loc-libre} how to adapt the proof of Theorems \ref{IMI}  and \ref{IMI-general} to the case of a locally free action. 

We again denote by $TY=\mathrm{Span}(K^M,K\in \g)$ the tangent bundle of the orbits in $U$. By Lemma \ref{bL-ND}, we have
\begin{equation}
\label{decompo-TU}
TU = TY \oplus (TY)^{\perp_{b^L}}.
\end{equation}
 Then we can choose the horizontal bundles of the fibrations $U\to B$ and $P\to M_G$ to be respectively
\begin{equation}
T^H U = (TY)^{\perp_{b^L}} \quad \text{and } \quad T^H P = \left.T^HU\right|_{P}\cap TP.
\end{equation}
Indeed, using \eqref{decompo-TU} and the fact that $TY\subset TP$, we see that 
\begin{equation}
\label{TP=TY+THP}
TP=TY\oplus T^H P.
\end{equation} 

Let $g^{T^HP}$ be a $G$-invariant and $J$-invariant metric on $T^H P$. Let $g^{TY}$ be a $G$-invariant metric on $TY$ and let $g^{JTY}$ be the $G$-invariant metric on $JTY$ induced by $J$ and $g^{TY}$. Then by \eqref{decompo-TU|P}, we can  chose the metric $g^{TM}$ on $TM$ so that on $P$:
\begin{equation}
\label{g^TM|_P}
g^{TM}|_P= g^{TY}|_P \oplus g^{JTY}|_P \oplus g^{T^HP}.
\end{equation}
We will use this condition on $g^{TM}$  in the rest of the introduction as well as in Sections \ref{Sect-rescaling}-\ref{Sect-action-loc-libre}.

Suppose that $U$ is small enough so that it can be identified with a $\e$-neighborhood, $\e>0$, of the zero section of the normal bundle $N $ of $P$ in $U$ via exponential map. We denote the corresponding  coordinate by $v  = (y,Z^\perp)\in U$ with $y\in P$ and $Z^\perp\in N_y$. Note that by \eqref{TP=TY+THP} and \eqref{g^TM|_P} we can identify $N_y$ and $JTY_y$.

Let $\boldrm{J}\in \End(TM|_P)$ be such that on $P$
\begin{equation}
\label{def-J}
\omega = g^{TM}(\boldrm{J}\cdot \,, \cdot).
\end{equation}

By \eqref{id-NG}, the normal bundle $N_G$ of $M_G$ in $B$ can be identified with the bundle $(JTY)_{B}$ induced on $B$ by $JTY\simeq N$ (see Section \ref{Sect-connections}). In particular, if $\pi(y)=x$, we keep the same notation for an element of $N_{y}$ and the corresponding element in $N_{G,x}$. 

We will see in Section \ref{Sect-calcul} that $\boldrm{J}$ intertwines  the bundles $TY$ and $JTY$. In particular, $\boldrm{J^2}$ induces an endomorphism of $N_G$ and if $\{a_1^\perp,\dots,a_d^\perp\}=-2\ic \pi\Sp(\boldrm{J}|_{(TY\oplus JTY)^{(1,0)}})$ ($a_j^\perp\in\R^*$), then
\begin{equation}
\Sp(\boldrm{J^2}|_{N_G}) = -\frac{1}{4\pi^2} \{a_1^{\perp,2},\dots,a_d^{\perp,2}\}.
\end{equation}

Let $g^{TB}$ be the metric on $TB$ induced by $g^{TM}$ and $T^H P$. Let $g^{N_G}$ be the induced metric on $N_G$ and $dv_{N_G}$ the corresponding volume form. For $x\in M_G$,  let $\{e_i^\perp\}_{i=1}^d$ be an orthonormal basis of $N_{G,x}$ such that $\boldrm{J}^2_x e_i^\perp= -\frac{1}{4\pi^2}a_i^{\perp,2}(x)e_i^\perp$. We can then identify $\R^d$ with $N_{G,x}$ via the map
\begin{equation}
(Z_1^\perp,\dots,Z_d^\perp) \in \R^d \mapsto Z^\perp=\sum_{i=1}^d Z_i^\perp e_i^\perp.
\end{equation}
We now define the operator $\LL^\perp_x$ acting on $N_{G,x} \simeq \R^d$ by
\begin{equation}
\label{def-Lperp}
\LL_x^\perp = -\sum_{i=1}^{d}\left((\n_{e_i^\perp})^2-|a_i^\perp Z_i^\perp|^2\right) - \sum_{j=1}^{d}a_j^\perp
\end{equation}
where  $\n_U$ denotes the ordinary differentiation operator on $\R^d$ in the direction $U$. We denote by $e^{-u\LL_x^\perp}(Z^\perp,Z'^\perp)$ the heat kernel of $\LL_x^\perp$ with respect to $dv_{N_{G,x}}(Z'^\perp)$. Note that we have an explicit formula for $e^{-u\LL_x^\perp}(Z^\perp,Z'^\perp)$ (see \eqref{noyau-Lperp(Zperp,Zperp)}), but we do not give it to have a simpler asymptotic formula for the heat kernel.

Let $g^{TM_G}$ be the metric on $M_G$ induced by $g^{TM}$ and $T^H P$ and $dv_{M_G}$ the corresponding volume form. We denote by $\langle \cdot\,,\cdot \rangle_G$ the $\C$-bilinear extension of $g^{TM_G}$ on $TM_G \otimes \C$. Then we can identify $R^{L_G}$ with the Hermitian matrix $\dot{R}^{L_G} \in \End (T^{(1,0)}M_G) $ such that for $V,V' \in T^{(1,0)}M_G$,
\begin{equation}
R^{L_G}(V,V') = \langle \dot{R}^{L_G} V,\ol{V'}\rangle_G.
\end{equation}

Let  $\{w_j\}$ be a local orthonormal frame of $T^{(1,0)}M$ with dual frame $\{w^j\}$. Set
\begin{equation}
\omega_{d} = -\sum_{i,j} R^{L}(w_i,\bw_j) \bw^j\wedge i_{\bw_i}.
\end{equation}
Let $h$ be the $G$-invariant smooth function on $M$ given by (see Section \ref{Sect-connections})
\begin{equation}
h(x) = \sqrt{\mathrm{vol}(G.x)},
\end{equation}
and let $\kappa \in \smooth(TB|_{M_G})$ be the function defined by $\kappa|_{M_G}=1$ and for $x\in M_G$, $Z\in T_{x}B$,
\begin{equation}
dv_{B}(x,Z) = \kappa(x,Z)dv_{T_xB}(Z) = \kappa(x,Z)dv_{M_G}(x)dv_{N_{G,x}}(Z).
\end{equation}

The following result is a version of \cite[Thm. 2.21]{ma-zhang}  in our situation for the heat kernel.
\begin{thm}
\label{thm-limit-near}
Assume that $G$ acts freely on $P$. For any fixed $u>0$ and $m \in \N$, we have the following convergence as $p\to +\infty$ for $|Z^\perp|<\e$:
 \begin{multline}
h(y,Z^\perp)^2\big(P_Ge^{-\frac{u}{p}D_p^2}P_G\big)\big((y,Z^\perp),(y,Z^\perp)\big) = \\
\frac{\kappa^{-1}(x,Z^\perp)}{(2\pi)^{n-d}} \frac{\det(\dot{R}_x^{L_G})e^{2u\omega_{d}(x)}}{\det\big(1-\exp(-2u\dot{R}_x^{L_G})\big)} e^{-u\LL^\perp_x}(\sqrt{p}Z^\perp,\sqrt{p}Z^\perp) \otimes \Id_E p^{n-d/2} \\
+ O\big(p^{n-d/2-1/2}(1+\sqrt{p}|Z^\perp|)^{-m}\big),
\end{multline}
 where $x=\pi(y)\in M_G$ and the term $O(\cdot)$ is uniform. The convergence is in the $\mathscr{C}^\infty$-topology in $y\in P$. Here, we use the convention that if an eigenvalue of $\dot{R}^{L_G}_{x_0}$ is zero, then its contribution to $\frac{\det(\dot{R}^{L_G}_{x_0})}{\det \big(1-\exp(-u\dot{R}^{L_G}_{x_0})\big)}$ is $\frac{1}{2u}$.
\end{thm}

From Theorems \ref{Euler-et-chaleur}, \ref{thm-estimate-away} and \ref{thm-limit-near}, we get Theorem \ref{IMI} in the case where $G$ acts freely on $P$ by integrating on $M$ the trace of $\big(P_Ge^{-\frac{u}{p}D_p^2}P_G\big)(m,m)$, then taking the limit $u\to +\infty$.

This paper is organized as follows. In Section \ref{Sect-connections}, we recall some constructions associated with a principal bundle. In Section \ref{Sect-reduction-laplacian}, we apply the constructions and results of Section \ref{Sect-connections} to our situation to define the reduction of $M$ and to descend the different objects we are given on it, thus proving Theorem~\ref{descent}. In Section \ref{Sect-localization} we prove the localization of the heat kernel near $P$, i.e., Theorem \ref{thm-estimate-away}. In Sections \ref{Sect-AD-heat-kernel}, we assume for simplicity that $G$ acts freely on $P$ and $\ol{U}$, and study the asymptotic of the heat kernel near $P$ by localizing the problem and studying a rescaled Laplacian on $B$. We thus obtain Theorem \ref{thm-limit-near}. Finally, in Section \ref{Sect-proof-IMI}, we prove the $G$-invariant holomorphic Morse inequalities (Theorems \ref{IMI} and \ref{IMI-general}) and  we also show how to use Theorem \ref{IMI-general} to get estimates on the other isotypic components of the cohomology $H^\bullet(M,L^p\otimes E)$.


\section{Connections and Laplacians associated with a principal bundle}
\label{Sect-connections}

In this section, we review some results of \cite[Chp. 1]{ma-zhang} for the convenience of the reader. 

Let $G$ be a connected compact Lie group of dimension $d$ that acts smoothly and locally freely on the left on a smooth manifold $M$ of dimension $m$. Then $\pi \colon M\to B=M/G$ is a $G$-principal bundle and $B$ is an orbifold. We denote by $TY$ the relative tangent bundle of this fibration.

Note that in \cite[Chp. 1]{ma-zhang}, Ma and Zhang assumed that $G$ acts freely on $M$, but as they explain in the introduction of \cite[Chp. 1]{ma-zhang} and in \cite[Sect. 4.1]{ma-zhang}, the results of \cite[Chp. 1]{ma-zhang} extend to the case where $G$ acts only locally freely, essentially because when we work on orbifold quotients, we in fact work with invariant sections on $M$.

Let $g^{TM}$ be a $G$-invariant metric on $TM$, and $\n^{TM}$ the corresponding Levi-Civita connection on $TM$. We denote by $T^HM$ the orthogonal complement of $TY$ in $TM$. For $U\in TB$, we denote by $U^H$ the horizontal lift of $U$ in $TM$, that is $\pi_* U^H = U$ and $U^H \in T^HM$. Let $\theta \colon TM\to \g$ be the connection form corresponding to $T^HM$, and let $\Theta$ be its curvature, i.e., the horizontal form such that
\begin{equation}
\label{def-Theta}
\Theta(U^H,V^H)=-P^{TY}[U^H,V^H],
\end{equation}
where $P^{TY}$ is the natural projection $TM=TY\oplus T^HM\to TY$.

The metric  $g^{TM}$ induces a metric $g^{TY}$ (resp. $g^{T^HM}$) on $TY$ (resp. $T^HM$). Let $g^{TB}$ be the metric on $TB$ induced by $g^{T^HM}$, and let $\n^{TB}$ be the corresponding Levi-Civita connection.

Let $(F,h^F)$ be a $G$-equivariant  Hermitian vector bundle with $G$-equivariant Hermitian connection $\n^F$. Then $G$ acts on $\smooth(M,F)$ by $(g.s)(x)=g.s(g^{-1}x)$.

Any $K\in \g$ induces a vector field $K^M$ on $M$ given by
\begin{equation}
\label{def-KM}
K^M_x=\left.\derpar{}{s}\right|_{s=0}e^{-sK}.x.
\end{equation}
For $K\in \g$, recall that $\lie_K$ is the infinitesimal action of $K$ on any $G$ representation. Let $\mu^F \in \smooth (M, \g^*\otimes \End(F))$ be defined by
\begin{equation}
\label{def-muF}
\mu^F(K) = \n^F_{K^M} - \lie_K.
\end{equation}
Using the identification $TY \simeq M\times \g$, we can identify $\mu^F$ with  $\wt{\mu}^F \in \smooth(M, TY\otimes \End(F))^G$ such that
\begin{equation}
\label{def-mutildeF}
\langle \wt{\mu}^F, K^M\rangle  = \mu^F(K).
\end{equation}

Let $F_B$ be the orbifold bundle on $B$ induced by $F$, i.e., $F_{B,x}=\smooth(\pi^{-1}(x),F|_{\pi^{-1}(x)})^G$. Then there is a canonical isomorphism
\begin{equation}
\pi_G \colon \smooth(M,F)^G \isom \smooth(B,F_B).
\end{equation}
The invariant metric $h^F$ induces a metric $h^{F_B}$ on $F_B$. For $s\in \smooth(B,F_B)$ and $U\in TB$, we define
\begin{equation}
\label{def-connection-down}
\n^{F_B}_Us := \n^F_{U^H} s.
\end{equation}
Observe that $\n^{F_B}$ is the restriction of the connection $\n^F - \mu(\theta)$ to $\smooth(M,F)^G$. Let $R^{F_B}$ be the curvature of $\n^{F_B}$. Then by  \cite[(1.18)]{ma-zhang} we have for $V,V'\in TB$
\begin{equation}
\label{formule-RFB}
R^{F_B} (V,V') = R^F(V^H,V'^H)-\mu^F(\Theta)(V,V').
\end{equation}

Let $dv_M$ be the Riemannian volume on $(M,g^{TM})$. We endow $\smooth(M,F)$ with the $L^2$ product induced by $g^{TM}$ and $h^F$:
\begin{equation}
\label{def-L^2-product}
\langle s,s' \rangle  = \int_M \langle s,s' \rangle_{h^F}(x)dv_M(x).
\end{equation}
In the same way, $g^{TB}$ and $h^{F_B}$ induce a $L^2$ product $\langle \cdot\,,\cdot \rangle$ on $\smooth(B,F_B)$.

For $x\in M$, we denote by $\mathrm{vol}(G.x)$ the  volume of the orbit of $G.x$ endowed with the restriction of $g^{TM}$. Define the $G$-invariant function $h$ on $M$ by
\begin{equation}
\label{def-h}
h(x) = \sqrt{\mathrm{vol}(G.x)}.
\end{equation}
Then $h$ define a function on $B$, which is still denoted by $h$. Note that $h$ is smooth only on the regular part of $B$.  However, we can extend it continuously to get a smooth function $\wh{h}$ on $B$. Then $\wh{h}$ also define a smooth function on $U$.

The map 
\begin{equation}
\label{def-Phi}
\Phi:= \wh{h}\pi_G \colon \big(\smooth(M,F)^G,\langle \cdot\,,\cdot \rangle\big) \to \big(\smooth(B,F_B),\langle \cdot\,,\cdot \rangle\big)
\end{equation}
is then an isometry.

Let $\{u_i\}_{i=1}^m$ be an orthonormal frame of $TM$. For any Hermitian bundle with Hermitian connection $(E,h^E,\n^E)$ on $M$, the Bochner Laplacians $\Delta^E, \Delta_M$ are given by
\begin{equation}
\Delta^E = -\sum_{i=1}^m \left( (\n^E_{u_i})^2 - \n^E_{\n^{TM}_{u_i}u_i} \right) \:, \quad \Delta_M = \Delta^\C.
\end{equation}

Let $\{f_l\}_{l=1}^d$be a $G$-invariant orthonormal frame of $TY$ with dual frame $\{f^l\}_{l=1}^d$, and let $\{e_i\}_{i=1}^{m-d}$ be an orthonormal frame of $TB$. Then $\{e_i^H,f_l\}$ form an orthonormal frame of $TM$.

For $\sigma,\sigma' \in TY\otimes \End(F)$, let $\langle \sigma,\sigma' \rangle_{g^{TY}} \in \End(F)$ be the contraction of the part of $\sigma\otimes \sigma'$ in $TY\otimes TY$ with $g^{TY}$. Note that
\begin{equation}
\label{carre-norme-mutilde-F}
\langle \wt{\mu}^F,\wt{\mu}^F \rangle_{g^{TY}} = \sum_{l=1}^d \langle \wt{\mu}^F,f_l \rangle_{g^{TY}}^2 \in \End(F).
\end{equation}

\begin{thm}
\label{induced-operator-on-B}
 As an operator on $\smooth(B,F_B)$, $\Phi \Delta^F\Phi^{-1}$ is given by
 \begin{equation}
\Phi \Delta^F\Phi^{-1} = \Delta^{F_B} - \langle \wt{\mu}^F,\wt{\mu}^F \rangle_{g^{TY}} -\wh{h}^{-1} \Delta_B \wh{h}.
\end{equation}
\end{thm}

\begin{proof}
 This is proved in  \cite[Thm. 1.3]{ma-zhang}.
\end{proof}


\section{The reduction of $M$ and the Laplacian on $B$}
\label{Sect-reduction-laplacian}

This Section is organized as follows. In Section \ref{Sect-reduction}  we apply the constructions and results of Section \ref{Sect-connections} to our situation in order to define the reduction of $M$ and to descend the different objects we are given on it. We prove, under Assumption \ref{assum-regular-value}, some properties of the reduction that are well-known in the case where $\omega$ is positive and get Theorem~\ref{descent}.  In Section \ref{Sect-operator}, we compute the operator induced on $U/G$ by the Kodaira Laplacian.

We use here the notations of the introduction. In particular, let $(M,J)$ be a connected compact complex manifold of dimension $n$, let $(L,h^L)$ be a holomorphic Hermitian line bundle on $M$ and $(E,h^E)$ a Hermitian complex vector bundle on $M$. We denote the associated Chern curvatures by $R^L$ and $R^E$. Let $\omega= \frac{\ic}{2\pi} R^L$ be the first Chern form of $(L,h^L)$, which is not assumed to be positive. Let  $G$ be a connected compact Lie group with Lie algebra $\g$. Let $d=\dim_\R G$. We assume that $G$ acts holomorphically on $(M,J)$, and that the action lifts in a holomorphic action on $L$ and $E$. We assume that $h^L$ and $h^E$ are preserved by the $G$-action.

Recall that $\n^L$ denotes the Chern connection of $(L,h^L)$ and that the moment map $\mu$ is defined by $2i\pi\mu(K)=\n^L_{K^M}-\lie_K$ for $K\in \g$. Let $P=\mu^{-1}(0)$ and $U$ a small tubular neighborhood of~$P$. Finally, we set $M_G=P/G$.

\subsection{The reduction of $M$}
\label{Sect-reduction}

 We begin by proving the following result.
\begin{lemme}
 \label{mu=moment-thm}
 The map $\mu$ is  smooth on $M$ and is linear in $K$. Moreover, it is moment map of the $G$-action on $M$, i.e., $\mu$ is $G$-equivariant and for any $K\in \g$,
\begin{equation}
\label{mu=moment}
d\mu (K) = i_{K^M} \omega.
\end{equation}
\end{lemme}

\begin{proof}
First, as both $\n^L_{K^M}$ and $\lie_K$ satisfies the Leibniz rules and preserves $h^L$, we know that $\n^L_{K^M}-\lie_K$ is $\smooth(M)$-linear, and moreover it is a skew-adjoint operator. Thus, under the canonical isomorphism $\End(L)=\C$,
\begin{equation}
\n^L_{K^M}-\lie_K \in \smooth(M,i \R).
\end{equation}
This proves the first part of Lemma \ref{mu=moment-thm}.

As $\n^L$ is $G$-invariant, we have $g\cdot(\n^L_Ys)=\n^L_{g_*Y} (g\cdot s)$ for $Y\in \smooth(M,TM)$, $s\in \smooth(M,L)$ and $g\in G$. Thus, taking $g=e^{-tK}$ for $K\in \g$ and differentiating at $t=0$, we get
\begin{equation}
\label{nablaL-et-LK-com}
\lie_K \n^L_Y s = \n^L_{[K^M,Y]}s+\n^L_Y\lie_Ks , \qquad \text{that is} \quad [\lie_K,\n^L]=0. 
\end{equation}

Using the definition of $\mu$ \eqref{def-mu}, \eqref{nablaL-et-LK-com} becomes
\begin{equation}
\big( 2i\pi \mu(K)+\n^L_{K^M} \big)\n^L_Y s = \n^L_{[K^M,Y]}s+\n^L_Y\big( 2i\pi \mu(K)+\n^L_{K^M} \big)s.
\end{equation}
This, together with \eqref{prequantization}, yields to 
\begin{equation}
Y(\mu(K))= \omega(K^M,Y),
\end{equation}
which is \eqref{mu=moment}

Finally, it is easy to prove that $g_*K^M = (\Ad_g K)^M$ and $g\cdot(\lie_Ks) = \lie_{\Ad_g K}(g\cdot s)$, so
\begin{equation}
2i\pi \, g\cdot (\mu(K)s) = \big(\n^L_{(\Ad_g K)^M}-\lie_{\Ad_g K}\big)(g\cdot s),
\end{equation}
and thus
\begin{equation}
\mu(g^{-1}x) = \Ad^*_{g^{-1}}\mu.
\end{equation}

The proof of Lemma \ref{mu=moment-thm} is complete.
\end{proof}

\begin{lemme}
\label{action-loc-libre}
The group $G$ acts locally freely on $P$.
\end{lemme}

\begin{proof}
 By \eqref{mu=moment}, we have for $x\in P$, $V\in T_xM$ and $K\in \g$,
 \begin{equation}
\omega(K^M,V)_x = \big(d_x\mu(V)\big)(K).
\end{equation}
In particular, if $K^M=0$, then $\big(d_x\mu(V)\big)(K)$ vanishes for all $V\in T_xM$. However, by Assumption~\ref{assum-regular-value}, the differential $d_x\mu\colon T_xM \to \g^*$ is surjective, hence $K=0$.
\end{proof}

\begin{lemme}
\label{bL-ND} 
When restricted to $TY\times TY$, the bilinear form $b^L$ is non-degenerate on $P$.
\end{lemme}

\begin{proof}
 First, observe that for $x\in P$, $V\in T_xM$ and $K\in \g$, equations  \eqref{def-b^L} and \eqref{mu=moment} yield
 \begin{equation}
 \label{b^L-and-dmu}
b^L(K^M,JV)_x = -\omega(K^M,V)_x = -\big(d_x\mu(V)\big)(K) .
\end{equation}

Let $x\in P$, $V\in T_xM$ and $K\in \g$. Then by  \eqref{b^L-and-dmu}
\begin{equation}
\label{b^L(K^m,JV)}
JV \in \left.(TY)^{\perp_{b^L}}\right|_{P} \iff d_x\mu(V) = 0 \iff V\in TP,
\end{equation}
the last equivalence coming from the fact that $P=\mu^{-1}(0)$. In particular, $\dim (TY)^{\perp_{b^L}} = \dim TP = 2n-d$, the last identity coming from the fact that $0$ is a regular value of $\mu$. Moreover, $\dim TY = d$ (because $G$ acts locally freely on $U$) and $TY+(TY)^{\perp_{b^L}}=TU$. This is possible only if this sum is direct, i.e., $TY\cap (TY)^{\perp_{b^L}}=\{0\}$. We have proved our lemma.
\end{proof}

By Lemma \ref{bL-ND}, we have
\begin{equation}
\label{decompo-TU-2}
TU = TY \oplus (TY)^{\perp_{b^L}}.
\end{equation}
 Then we can choose the horizontal bundles of the fibrations $U\to B$ and $P\to M_G$ to be
\begin{equation}
\label{def-espaces-horizontaux}
T^H U = (TY)^{\perp_{b^L}} \quad \text{and } \quad T^H P = \left.T^HU\right|_{P}\cap TP.
\end{equation}
Indeed, using \eqref{decompo-TU-2} and the fact that $TY\subset TP$, we see that 
\begin{equation}
\label{TP=TY+THP-2}
TP=TY\oplus T^H P.
\end{equation} 

Let $(L_G,h^{L_G},\n^{L_G})$ and $(E_G,h^{E_G},\n^{E_G})$ be defined from $(L,h^L,\n^L)$, $(E,h^E,\n^E)$ and $T^HP$ as indicated in Section \ref{Sect-connections}. We also define $\omega_G$ by
\begin{equation}
\label{def-omegaG}
\omega_G(V,V') = \omega (V^H,V'^H).
\end{equation}

Note that \eqref{formule-RFB} restricted to $P=\mu^{-1}(0)$ gives
\begin{equation}
\label{courbure-sur-MG}
R^{L_B}|_{M_G} (V,V') = R^L|_{P}(V^H,V'^H).
\end{equation}
From \eqref{prequantization}, \eqref{def-omegaG} and \eqref{courbure-sur-MG}, we see that if $R^{L_G}$ is the curvature of $\n^{L_G}$, then 
\begin{equation}
\label{prequantization_G}
\omega_G = \frac{\ic}{2\pi} R^{L_G}.
\end{equation}

\begin{lemme}
\label{TY^perp=JTP}
We have
 \begin{equation}
 \label{T^HU=JTP-eq} 
 \begin{aligned}
& \left.T^HU\right|_{P} = JTP \\
&TU |_P= TP \oplus JTY.
\end{aligned}
\end{equation}
In the second line, the sum is orthogonal with respect to $b^L$.
\end{lemme}

\begin{proof}
Recall that $T^HU  = (TY)^{\perp_{b^L}}$. Thus the first identity in \eqref{T^HU=JTP-eq}  follows from \eqref{b^L(K^m,JV)}. 

Concerning the second, we have  for $V\in TP$ and $K\in \g$,
 \begin{equation}
 \label{b^L(JK^m,V)}
b^L(JK^M,V)_x = \omega(K^M,V)_x =\big(d_x\mu(V)\big)_x(K)=0 .
\end{equation}
Using \eqref{b^L(JK^m,V)} and the facts  that $b^L$ is non-degenerate on $JTY$ and that $\dim TU = \dim TP +\dim JTY$, we get the second identity in \eqref{T^HU=JTP-eq}.
\end{proof}

Using Lemma \ref{TY^perp=JTP} and \eqref{TP=TY+THP-2}, we find firstly
\begin{equation}
\label{decompo-TU|P}
TU |_P= T^HP\oplus TY \oplus JTY,
\end{equation} 
the decomposition being orthogonal for $b^L$, and secondly by Lemma \ref{TY^perp=JTP} and \eqref{def-espaces-horizontaux},
\begin{equation}
\label{THP=TP-inter-JTP}
T^HP = TP\cap JTP.
\end{equation}
In particular, $T^HP$ is stable by $J$, so we can define an almost-complex structure on $M_G$ in the following way. For $V\in TM_G$, we denote $V^H$ its lift in $T^HP$, and we define the almost complex structure $J_G$ on $M_G$ by
\begin{equation}
\label{def-J_G}
(J_GV)^H  = J (V^H).
\end{equation}

\begin{lemme}
\label{J_G-integrable}
The almost complex structure  $J_G$ is integrable, thus $(M_G,J_G)$ is a complex manifold.
\end{lemme}

\begin{proof}
  Let $u,v \in \smooth (M_G,T^{1,0}M_G)$. Then there are $U,V \in \smooth (M_G,T M_G)$ such that
\begin{equation}
u=U-\ic J_GU \: , \quad v=V-\ic J_GV.
\end{equation}
Using \eqref{def-J_G}, we find
\begin{equation}
u^H=U^H-\ic JU^H \: , \quad v=V^H-\ic JV^H \in T^{1,0}M \cap T_\C P.
\end{equation}
As both $T^{1,0}M$ and $T_\C P$ are integrable, we have $[u^H,v^H] \in T^{1,0}M \cap T_\C P$, i.e., there is $W\in \smooth(M,TM)$ such that
\begin{equation}
[u^H,v^H]=W-\ic JW,
\end{equation}
and moreover $W,JW \in TP$. Thus, $W\in TP\cap JTP= T^HP$ and we can write $W = X^H$ for $X$ a section of $TM_G$. Hence
\begin{equation}
\label{[u,v]}
[u,v] = \pi_*[u^H,v^H] = \pi_*(X^H-\ic JX^H) = X-\ic J_GX \in  T^{1,0}M_G.
\end{equation}
By the Newlander-Nirenberg theorem, \eqref{[u,v]} means that  $J_G$ is integrable.
\end{proof}

\begin{lemme}
\label{L_G-holomorphic}
The bundles $L_G$ and $E_G$ are holomorphic. Moreover, $\n^{L_G}$ and $\n^{E_G}$ are the respective Chern connections of $(L_G,h^{L_G})$ and $(E_G,h^{E_G})$.
\end{lemme}

\begin{proof}
We first prove the result for $L_G$. 

Observe that for $U,V\in TM_G$,
\begin{equation}
\omega_G(J_GU,J_GV) = \omega(JU^H,JV^H)=\omega(U^H,V^H) = \omega_G(U,V).
\end{equation}
Hence, $\omega_G$ is a $(1,1)$-form, and so is $R^{L_G}$ by \eqref{prequantization_G}. We decompose $\n^{L_G}$ into holomorphic part and anti-holomorphic part,
\begin{equation}
\n^{L_G} = (\n^{L_G})^{1,0}+(\n^{L_G})^{0,1}.
\end{equation}
As $R^{L_G}$ is $(1,1)$, we have
\begin{equation}
\label{R^(0,2)}
\big((\n^{L_G})^{0,1}\big)^2 =0.
\end{equation}

For $s\in \smooth(M_G,L_G)$, we define
\begin{equation}
\label{def-db^L_G}
\db^{L_G} s = (\n^{L_G})^{0,1}s.
\end{equation}
Let $s_0$ be a local frame of $L_G$ near $x_0\in M_G$. Then we can write $(\n^{L_G})^{0,1}s_0 = \alpha s_0$ for some $(0,1)$-form $\alpha$. By \eqref{R^(0,2)}, we have
\begin{equation}
0= \big((\n^{L_G})^{0,1}\big)^2s_0 = (\db\alpha)s_0.
\end{equation}
Thus, $\db \alpha = 0$. By the (local) $\db$-lemma, there is a function $f$ defined near $x_0$ such that $\db f = - \alpha$. Thus,
\begin{equation}
\db^{L_G} s_0 +(\db f) s_0 =0.
\end{equation}
This shows that \eqref{def-db^L_G} defines a holomorphic structure on $L_G$, for wich $e^f s_0$ is a local holomorphic frame near $x_0$. 

Finally, $\n^{L_G}$ is clearly Hermitian with respect to $h^{L_G}$, and is holomorphic by the definition  \eqref{def-db^L_G}, so $\n^{L_G}$ is indeed the Chern connection on $L_G$.

We now turn to $E_G$. Here again, it is enough to prove that $R^{E_G}$ is a $(1,1)$-form (see for instance \cite[Prop. I.3.7]{MR909698}).  As $R^E$ is a $(1,1)$-form,  \eqref{formule-RFB} shows that it is equivalent to prove that $\Theta|_{T^HP\times T^HP}$ is a $(1,1)$-form. 

Let $u=U-\ic JU$ and  $v=V-\ic JV$ be in $(T^HP)^{1,0}$. As $U,V,JU$ and $JV$ are in $T^HP=TP\cap JTP$ and $TP$ is integrable, we have $[u,v]\in T_\C P$. Moreover, as  $u$ and $v$ are of type $(1,0)$ and $J$ is integrable,  $[u,v]$ is also of type $(1,0)$, and thus $[u,v]=-iJ[u,v]\in JT_\C P$. In conclusion, $[u,v]\in T^H_\C P$ and by \eqref{def-Theta}, $\Theta(u,v)=0$.
\end{proof}

\begin{lemme}
\label{ker-omega_G}
We have $\ker \omega|_P \subset T^HP$, and for $x\in M_G$, $\pi_*$ induces an isomorphism
\begin{equation}
\label{noyau-descente}
(\ker \omega_G)_x \simeq (\ker \omega)|_{\pi^{-1}(x)}.
\end{equation}\end{lemme}

\begin{proof}
 Let $V\in TU|_P$ be such that $\omega(V,\cdot)=0$. Then we also have $b^L(V,\cdot)=0$. Thus $V$ is  in particular in $(TY)^{\perp_{b^L}}=T^HU$. Moreover, $V$ is also orthogonal (for $b^L$) to $JTY$, so is in $T^HP$ by Lemma \ref{TY^perp=JTP}. 
 
 As $\omega_G(\pi_*\cdot \,, \pi_*\cdot) = \omega(\cdot \,,\cdot)$, we know that $\pi_*$ maps $\ker \omega|_P$ in $\ker \omega_G$, and is injective as $\ker \omega|_P \subset T^HP$. Finally, if $V\in \ker \omega_G$, then $\omega(V^H, V') = 0$  for $V'\in T^HP$. In fact, as the decomposition in \eqref{decompo-TU|P} is orthogonal for $b^L$, we have $\omega(V^H, V') = 0$  for $V'\in TU|_P$, and thus $V^H \in \ker \omega$. The proof of our lemma is complete.
\end{proof}

By  \eqref{prequantization_G} and Lemmas \ref{J_G-integrable}, \ref{L_G-holomorphic} and \ref{ker-omega_G} we have proved Theorem \ref{descent}.


\subsection{The Kodaira Laplacian and the operator induced on $B$}
\label{Sect-operator}

We define the vector bundle $\mathcal{E}$, and $\E_p$ ($p\geq1$) over $M$ by
\begin{equation}
\begin{aligned}
&\mathcal{E}= \Wedge(T^*M)\otimes E. \\
&\E_{p}= \Wedge(T^*M)\otimes E \otimes L^p.
\end{aligned}
\end{equation}

 Recall that $g^{TM}$ is a $J$- and $G$-invariant metric on $TM$ (we do not assume that \eqref{g^TM|_P} holds in this section). We endow $\smooth(M,\E_p)$ with the $L^2$ scalar product associated with $g^{TM}$, $h^L$ and $h^E$ as in \eqref{def-L^2-product}. Then the Dolbeault-Dirac operator $D_p$ defined in \eqref{def-D_p} is a  formally self-adjoint operator acting on $\smooth(M,\E_p)$.

We now recall the Lichnerowicz formula for the Kodaira Laplacian $D_p^2$.

Let $\n^{TM}$ be the Levi-Civita connection on $(M,g^{TM})$. We denote by $P^{T^{(1,0)}M}$ the orthogonal projection form $TM\otimes_\R \C$ onto $T^{(1,0)}M$. Let $\n^{T^{(1,0)}M}= P^{T^{(1,0)}M}\n^{TM}P^{T^{(1,0)}M}$ be the induced connection on $T^{(1,0)}M$. We endow $\det(T^{1,0}M)$ with the metric induced by $g^{TM}$, and we denote by  $\n^{\det}$ the Hermitian connection on $\det(T^{1,0}M)$  induced by $\n^{T^{(1,0)}M}$. Let $R^{\det}$  be the curvature of $\n^{\det}$.

 Let $(w_1,\dots,w_n)$ be an orthonormal frame of $(T^{(1,0)}M,g^{TM})$, and $(e_1,\dots,e_{2n})$ be the orthonormal frame of $(TM,g^{TM})$ given by
\begin{equation}
e_{2j-1}=\frac{1}{\sqrt{2}}(w_j+\bw_j) \quad \text{ and }\quad e_{2j}=\frac{\ic}{\sqrt{2}}(w_j-\bw_j).
\end{equation}
Let $\{e^k\}$ be the dual basis of $\{e_k\}$.
The Clifford action of $T^*_\C M$ on $\Wedge(T^*M)$ is defined by linearity from
\begin{equation}
c(w_j):= \sqrt{2}\bw^j\wedge  \quad\text{and} \quad c(\bw_j):=-\sqrt{2}i_{\bw_j}.
\end{equation}
We then define a map, still denoted by $c(\cdot)$, on $\Lambda(T^*_\C M)$ by setting for $j_1<\dots<j_k$:
\begin{equation}
c(e^{j_1}\wedge \dots \wedge e^{j_k}) := c(e_{j_1})\dots c(e_{j_k}).
\end{equation}

Let $\Gamma^{TM}$ and $\Gamma^{\det}$ be the connection forms of $\n^{TM}$ and $\n^{\det}$ associated with the frames $\{e_i\}$ and $w_1\wedge \dots \wedge w_n$. Define the the \emph{Clifford connection} on $\Wedge(T^*M)$ (see \cite[(1.3.5)]{ma-marinescu}) by the following local formula in the frame $\{\bw^{i_1}\wedge\dots\wedge \bw^{i_k}\}$:
\begin{equation}
\label{def-nablaCl}
\nabla^{\mathrm{Cl}}= d + \frac{1}{4}\sum_{i,j} \langle \Gamma^{TM}e_i,e_j \rangle c(e_i)c(e_j) + \frac{1}{2} \Gamma^{\det}.
\end{equation}
We also denote by $\n^{\mathrm{Cl}}$ the connection on $\mathcal{E}$ induced by $\n^{\mathrm{Cl}}$ and $\n^E$.

Let $\Omega$ be the real $(1,1)$-form defined by
 \begin{equation}
 \label{def-Omega}
\Omega = g^{TM}(J\cdot\,,\cdot).
\end{equation}
On $\Wedge(T^*M)$, we define the \emph{Bismut connection} $\n^{\mathrm{Bi}}$ by
\begin{equation}
\label{def-nablaB}
\n^{\mathrm{Bi}}_V = \n^{\mathrm{Cl}}_V +\frac{\ic}{4}c\big(i_V(\partial -\db)\Omega\big).
\end{equation}
This connection, along with $\n^E$ and $\n^L$, induces  connections $\n^{\mathcal{E}}$ and $\n^{\E_p}$ on $\mathcal{E}$ and $\E_p$. Moreover, we know that (see e.g. \cite[Thm. 1.4.5]{ma-marinescu})
\begin{equation}
\label{formule-Dp}
D_p = \sum_{i=1}^{2n}c(e_i)\n^{\E_p}_{e_i}.
\end{equation}

Let $\Delta ^{\E_p}$ is the Bochner Laplacian on $\E_p$ induced by  $\n^{\E_p}$. It is given by the following formula: if $(g^{ij})$ is the inverse of the matrix $(g_{ij})=(g^{TM}_Z(e_i,e_j))$, then
\begin{equation}
\label{def-laplacien-de-Bochner}
\Delta^{\E_p}= -g^{ij} \left(\n^{\E_p}_ {e_i}\n^{\E_p}_ {e_j} -\n^{\E_p}_ {\nabla ^{TM}_{e_i}e_j} \right).
\end{equation}

Let $r^M$ be the scalar curvature of $(M,g^{TM})$.  Let $\Psi_\mathcal{E}$ be the smooth self-adjoint section of $\End(\mathcal{E})$ given by
\begin{equation}
\Psi_\mathcal{E} = \frac{r^M}{4}+c\big(R^E +\frac{1}{2}R^{\det}\big) + \frac{\ic}{2}c(\db\partial \Omega)-\frac{1}{8} \big|(\partial -\db)\Omega\big|^2.
\end{equation}
 Set also
 \begin{equation}
\begin{aligned}
&\omega_d = -\sum_{i,j} R^{L}(w_i,\bw_j) \bw^j\wedge i_{\bw_i}, \\
&\tau = \sum_i R^L(w_i,\bw_i). \\
\end{aligned}
\end{equation}

The Lichnerowicz formula (see for instance \cite[Thm. 1.4.7 and (1.5.17)]{ma-marinescu}) reads
\begin{equation}
\label{Lich-D_p^2}
D_p^2 = \Delta ^{\E_p} -p(2\omega_d +\tau) + \Psi_\mathcal{E},
\end{equation}

Let $\mu^E$, $\mu^{\mathrm{Bi}}$ and $\mu^{\E_p}$ be the moment maps induced by $\n^E$, $\n^{\mathrm{Bi}}$ and $\n^{\E_p}$ as in \eqref{def-muF}. Recall that $\mu$ is defined in \eqref{def-mu}. Then we have
\begin{equation}
\label{relations-mus}
\left\{
\begin{aligned}
&\mu^L = 2i\pi \mu, \\
&\mu^{\E_p}= 2i\pi p\mu+\mu^E +\mu^{\mathrm{Bi}}.
\end{aligned}\right.
\end{equation}

Assume now that $G$ acts freely on $P$, and recall that we then choose the $G$-invariant neighborhood $U$ of $P$  so that $G$ acts freely on its closure $\ol{U}$. Using  the procedure of Section \ref{Sect-connections} for $U\to U/G=B$ and $g^{TM}|_U$, we can define the operator $\Phi D_p^2 \Phi^{-1}$ induced by $D_p^2$ on $B$. Thanks to Theorem \ref{induced-operator-on-B} and \eqref{Lich-D_p^2}, we find that in the case of a free $G$-action on $P$,
\begin{equation}
\label{Lich-PhiDp2Phi^-1}
\Phi D_p^2 \Phi^{-1} = \Delta ^{\E_{p,B}} -p(2\omega_d +\tau) + \Psi_\mathcal{E} -  \langle \wt{\mu}^{\E_p},\wt{\mu}^{\E_p} \rangle_{g^{TY}} -\wh{h}^{-1} \Delta_B \wh{h}.
\end{equation}
Here, we have kept the same notation for an element in $\smooth(U,\End(\E_p))^G$ and the induced element in $\smooth(B,\End(\E_{p,B}))$, and we will always do this in the sequel.


\section{Localization near $P$}
\label{Sect-localization}

The goal of this section is to prove the localization of $P_Ge^{-\frac{u}{p}D_p^2}P_G$ near $P$, i.e., we prove Theorem \ref{thm-estimate-away}.

Let $\mathrm{inj}^M$ be the injectivity radius of $(M,g^{TM})$, and $\e \in ]0,\mathrm{inj}^M[$.

For  $x_0\in M$, we denote by $B^M(x_0,\e)$ and $B^{T_{x_0}M}(0,\e)$ the open balls in $M$ and $T_{x_0}M$ with center $x_0$ and $0$ and radius $\e$ respectively. If $\exp^M_{x_0}$ is the exponential map of $M$, then  $ Z\in B^{T_{x_0}M}(0,\e) \mapsto \exp^M_{x_0}(Z) \in B^M(x_0,\e)$ is a diffeomorphism, which gives local coordinates by identifying $T_{x_0}M$ with $\R^{2n}$ via an orthonormal basis $\{e_i \}$ of $T_{x_0}M$:

\begin{equation}
\label{R2n=TM}
(Z_1,\dots,Z_{2n}) \in \R^{2n} \mapsto \sum_i Z_i e_i \in T_{x}M.
\end{equation} 
From now on, we will always identify $ B^{T_{x_0}M}(0,\e)$ and $ B^M(x_0,\e)$.

Let $x_1,\dots x_N$ be points of $M$ such that $\{U_k = B^M(x_k,\e)\}_{k=1}^N$ is an open covering of $M$. On each $U_k$ we identify $E_Z$, $L_Z$ and $\Wedge(T_Z^*M)$ to $E_{x_k}$, $L_{x_k}$ and $\Wedge(T_{x_k}^*M)$ by parallel transport with respect to $\n^E$, $\n^L$ and $\n^{\mathrm{Bi}}$ along the geodesic ray $t \in [0,1]\mapsto tZ$. We fixe for each $k=1,\dots,N$ an orthonormal basis $\{e_i\}_i$ of $T_{x_k}M$ (without mentioning the dependence on $k$).

We denote by $\nabla_V$ the ordinary differentiation operator in the direction $V$ on $T_{x_k}M$.

Let $\{\varphi_k\}_k$ be a partition of unity subordinate to $\{U_k\}_k$. For $\ell \in \N$, we define a Sobolev norm $||\cdot||_{\sob^\ell(p)}$ on the $\ell$-th Sobolev space $\sob^\ell (M,\E_p)$ by
\begin{equation}
\label{def-norme-H^l(p)}
||s||^2_{\sob^\ell(p)} = \sum_k \sum_{j=0}^\ell \sum_{i_1,\dots,i_j=1}^d ||\n_{e_{i_1}}\dots \n_{e_{i_j}}(\varphi_k s)||^2_{L^2}.
\end{equation}

\begin{lemme}
\label{estimee-elliptique(p)}
For any  $m\in \N$, there exists $C_m>0$  such that for any $p\in \N^*$ and any $s\in \sob^{2m+2}(M,\E_p)$,
\begin{equation}
\label{estimee-elliptique(p)-eq}
||s||^2_{\sob^{2m+2}(p)}\leq C_m  p^{4m+4} \sum_{j=0}^{m+1}p^{-4j}||D_p^{2j}s||_{L^2}.
\end{equation}
\end{lemme}

\begin{proof}
 This is proved in \cite[Lem. 1.6.2]{ma-marinescu}.
\end{proof}

Let $f\colon \R \to [0,1]$ be a smooth even function such that
\begin{equation}
\label{def-f}
f(t)=\left \{
\begin{aligned}
&1 \text{ for } |t|<\e/2, \\
& 0 \text{ for } |t|>\e.
\end{aligned}
\right.
\end{equation}

For $u>0$, $\varsigma\geq1$ and $a\in \C$, set
\begin{equation}
\label{defFuGuHu}
\begin{aligned}
&\F_u(a)=\int_\R e^{iv\sqrt{2}a}\exp(-v^2/2)f(v\sqrt{u})\frac{dv}{\sqrt{2\pi}}, \\
&\G_u(a)=\int_\R e^{iv\sqrt{2}a}\exp(-v^2/2)(1-f(v\sqrt{u}))\frac{dv}{\sqrt{2\pi}}.
\end{aligned}
\end{equation}

These functions are even holomorphic functions. Moreover, the restrictions of $\F_u$ and $\G_u$  to $\R$ lie in the Schwartz space $\mathcal{S}(\R)$, and
\begin{equation}
\label{liensFGH}
 \F_u(vD_p)+\G_u(vD_p)=\exp\left( -v^2D_p^2\right) \text{ for }v>0.
\end{equation}

Let $\G_u(vL_p)(x,x')$ be the smooth kernel of $\G_u(vL_p)$ with respect to $dv_M(x')$.
\begin{prop}
\label{the-pb-is-local}
For any $m\in \N$, $u_0>0$, $\e>0$, there exist  $C>0$ and $N\in \N$ such that for any $u> u_0$ and any $p\in \N^*$,
\begin{equation}
\label{the-pb-is-local-eq}
\left|  \G_{\frac{u}{p}} \left(\sqrt{u/p}D_p\right)(\cdot\, ,\cdot) \right|_{\mathscr{C}^m(M\times M)} \leq Cp^N  \exp \left( -\frac{\e^2p}{16u} \right).
\end{equation}
Here, the $\mathscr{C}^m$-norm is induced by $\n^L$, $\n^E$, $\n^{\mathrm{Bi}}$, $h^L$, $h^E$ and $g^{TM}$.
\end{prop}

\begin{proof}
This is proved in \cite[Prop. 1.6.4]{ma-marinescu}. 
\end{proof}

\begin{proof}[Proof of Theorem \ref{thm-estimate-away}]
As $0$ is a regular value of $\mu$, there is $\epsilon_0$ such that 
\begin{equation}
\mu \colon M_{2\epsilon_0}:=\mu^{-1}(B^{\g^*}(0,2\epsilon_0)) \to B^{\g^*}(0,2\epsilon_0)
\end{equation}
is a submersion. Note that $M_{2\epsilon_0}$ is an open $G$-invariant subset of $M$.

Fix $\e,\epsilon_0$ small enough so that $M_{2\epsilon_0}\subset U$ and $d^M(x,y)>4\e$ if $x\in M_{\epsilon_0}$ and $y\in M\setminus U$. We set $V_{\epsilon_0}=M\setminus M_{\epsilon_0}$, which is a smooth $G$-manifold with boundary $\partial V_{\epsilon_0}$. Then $M\setminus U \subset V_{\epsilon_0}$.

We denote by $D_{p,D}$ the operator $D_p$ acting on $V_{\epsilon_0}$ with the Dirichlet boundary condition. Then $D_{p,D}$ is self-adjoint.

By \cite[Sects. 2.6, 2.8]{MR1395148} and \cite[Append. D.2]{ma-marinescu}, we know that the wave operator $\cos(uD_{p,D})$ is well defined and its Schwartz kernel $\cos(uD_{p,D})(x,x')$ only depends on the restriction of $D_p$ to $G\cdot B^{M}(x,u)\cap V_{\epsilon_0}$ and vanish if $d^M(x,x')\geq u$. Thus, by \eqref{defFuGuHu}, 
\begin{equation}
\label{F(Dp)-et-F(DpD)}
\F_{\frac{u}{p}} \left(\sqrt{u/p}D_p\right)(x,x')=\F_{\frac{u}{p}} \left(\sqrt{u/p}D_{p,D}\right)(x,x') \qquad \text{if }x,x'\in M\setminus U.
\end{equation}

Let $s\in \smooth(M, \E_p)^G$ with $\supp(s)\subset \overset{\circ}{V}_{\epsilon_0}$. Since $D_p$ commutes with the $G$-action, we know that $D_ps\in \Omega^{0,\bullet}(M,L^p\otimes E)^G$. Moreover, from the Lichnerowicz formula \eqref{Lich-D_p^2}, we get
\begin{equation}
\label{<Dp2s,s>}
\langle D_p^2 s,s\rangle = \|\n^{\E_p}s\|_{L^2}^2 -p\langle(\omega_d+\tau)s,s\rangle +\langle \Psi_\mathcal{E} s,s\rangle.
\end{equation}
Observer that, as $s\in \Omega^{0,\bullet}(M,L^p\otimes E)^G$, \eqref{def-muF} gives
\begin{equation}
\n^{\E_p}_{K^M}s = (\lie_K + \mu^{\E_p}(K))s = \mu^{\E_p}(K)s,
\end{equation}
and thus by \eqref{relations-mus} and the fact that $\supp(s)\subset \overset{\circ}{V}_{\epsilon_0}$,
\begin{equation}
\label{norme-nablaEp-s1}
\begin{aligned}
\|\n^{\E_p}s\|_{L^2}^2 \geq &C\sum_i \|\n^{\E_p}_{K^M_i}s\|_{L^2}^2=C\sum_i \|\mu^{\E_p}(K_i)s\|_{L^2}^2 \\
&\geq Cp^2 \| \, |\mu|s\|_{L^2}^2 -C'\|s\|_{L^2}^2 \\
&\geq C\epsilon_0^2 p^2 \|s\|_{L^2}^2-C'\|s\|_{L^2}^2.
\end{aligned}
\end{equation}

Thanks to \eqref{<Dp2s,s>} and \eqref{norme-nablaEp-s1}, we have
\begin{equation}
\langle D_p^2 s,s\rangle \geq Cp^2 \|s\|_{L^2}^2.
\end{equation}
In particular, as $P_G$ preserve the Dirichlet boundary condition, there are $C,C'>0$ such that for $p\geq1$,
\begin{equation}
\label{sp-DpD2-G-inv}
\Sp(P_GD_{p,D}^2P_G)\subset [Cp^2,+\infty[\,.
\end{equation}

By the elliptic estimate for the Laplacian with Dirichlet boundary condition \cite[Thm. 5.1.3]{MR1395148}, we can see that the proof of Lemma \ref{estimee-elliptique(p)} (see \cite[Lem. 1.6.2]{ma-marinescu}) still works if we replace therein $D_p$ by $D_{p,D}$ and take $s\in \sob^{2m+2}(M,\E_p)\cap\sob^1_0(M,\E_p)$. Using this modification of Lemma \ref{estimee-elliptique(p)}, \eqref{sp-DpD2-G-inv} and
\begin{equation}
\sup_{a\geq Cp} |a^m\F_{\frac{u}{p}}(a\sqrt{u/p})| \leq C_{m,k,u} p^{-k},
\end{equation}
 we find  that for any $Q,Q$ differential operators of order $2m,2m'$ with scalar principal symbol and with support in $U_i,U_j$ and for any $k\in \N$
\begin{equation}
\label{estimationPG(LpD)Q}
\left\| Q \F_{\frac{u}{p}} \Big(\sqrt{u/p}D_{p,D}\Big)Q' s \right \|_{L^2} \leq C_{m,m',u} p^{-k} \|s\|_{L^2}.
\end{equation}
Thus, using Sobolev inequalities with \eqref{estimationPG(LpD)Q}, and  \eqref{liensFGH}, \eqref{the-pb-is-local-eq} and \eqref{F(Dp)-et-F(DpD)}, we get Theorem \ref{thm-estimate-away}.
\end{proof}

\section{Asymptotic of the heat kernel near $P$ for a free action}
\label{Sect-AD-heat-kernel}

We assume in this Section that $G$ acts freely on $P$ and $\ol{U}$.

In this section, we prove Theorem \ref{thm-limit-near}. In Section \ref{Sect-rescaling}, we work near $P$ and replace our geometric setting by a model setting, in which $M$ is replaced by $G\times \R^{2n-d}$, $P$ by $G\times \R^{2n-2d}\times \{0\}$ and the different bundles are trivial. We can then define a rescaled version of $\frac{1}{p}D_p^2$, and in Section \ref{Sect-convergence}, we prove the convergence of the heat kernel of the rescaled operator. In Section \ref{Sect-calcul} we compute the limiting heat kernel to finish the proof of Theorem \ref{thm-limit-near}.

\subsection{Rescaling the operator $\Phi D_p^2\Phi^{-1}$}
\label{Sect-rescaling}

 This section is analogous to \cite[Sect. 2.6]{ma-zhang}, with the necessary changes made.

We fix $x_0\in M_G$ and  $\e \in ]0,\mathrm{inj}^M/4[$.

Recall that we have the following diagram:
$$  \xymatrix{
 P=\mu^{-1}(0) \ar@{^{(}->}[r] \ar[d]^{G} & U \ar[d]^{G}  \\
 M_G \ar@{^{(}->}[r]& B
 }$$

Recall also that $g^{T^HP}$ is a $G$-invariant and $J$-invariant metric on $T^H P$, $g^{TY}$ is a $G$-invariant metric on $TY$ and  $g^{JTY}$ is the $G$-invariant metric on $JTY$ induced by $J$ and $g^{TY}$. Then by \eqref{decompo-TU|P}, we can chose  be a $G$-invariant metric $g^{TM}$ on $M$ such that on $P$:
\begin{equation}
g^{TM}|_P= g^{TY}|_P \oplus g^{JTY}|_P \oplus g^{T^HP}.
\end{equation}

Let $g^{T^HU}$  be the restriction of $g^{TM}$ on $T^HU$. Let $g^{TB}$ (resp. $g^{TM_G}$) be the metric on $TB$ (resp. $TM_G$) induced by $g^{T^HU}$ (resp. $g^{T^HP}$).

By \eqref{TP=TY+THP} and Lemma \ref{TY^perp=JTP}, we know that 
\begin{equation}
T^HU|_P=JTY|_P\oplus JT^HP = JTY|_P\oplus T^HP.
\end{equation}
 As a consequence, if $N_G$ denotes the normal bundle of $M_G$ in $B$, then $N_G$ can be identified as
\begin{equation}
\label{id-NG}
N_G\simeq(TM_G)^{\perp_{g^{TB}}}= (JTY)_B|_{M_G},
\end{equation}
where $(JTY)_B$ denotes the bundle over $B$ induced by $JTY$.

Let $\n^{TB}$ be the Levi-Civita connection on $(TB,g^{TB})$. Let $P^{N_G}$ and $P^{TM_G}$ be the orthogonal projections from $TB|_{M_G}$ to $N_G$ and $TM_G$ respectively. Set
\begin{equation}
\label{connection-sur-B}
\begin{aligned}
& \n^{N_G} = P^{N_G}(\n^{TB}|_{M_G})P^{N_G}, &&\n^{TM_G} = P^{TM_G}(\n^{TB}|_{M_G})P^{TM_G}, \\
&{}^0\n^{TB} = \n^{N_G}\oplus \n^{TM_G}, && A=\n^{TB}|_{M_G}-{}^0\n^{TB}.
\end{aligned}
\end{equation}

For $W\in T_{x_0}M_G$, let $u\in \R \mapsto x_u = \exp_{x_0}^{M_G}(uW) \in M_G$ be the geodesic in $M_G$ starting at $x_0$ with speed $W$. If $|W|\leq 4\e$ and $V\in N_{G,x_0}$, let $\tau_W V$ be the parallel transport of $V$ with respect to $\n^{N_G}$ along to curve $u\in [0,1] \mapsto x_u = \exp_{x_0}^{M_G}(uW)$.

If $Z\in T_{x_0}B$, we decompose $Z$ as $Z=Z^0+Z^\perp$ with $Z^0\in T_{x_0}M_G$ and $Z^\perp\in N_{G,x_0}$, and if $|Z^0|,|Z^\perp|\leq \e$ we identify $Z$ with $\exp^B_{\exp_{x_0}^{M_G}(Z^0)}(\tau_{Z^0}Z^\perp)$. This gives a diffeomorphism 
\begin{equation}
\mathcal{F} \colon B^{T_{x_0}M_G}(0,4\e)\times B^{N_{G,x_0}}(0,4\e) \isom \mathscr{U}(x_0) \subset B,
\end{equation}
where $\mathscr{U}(x_0)$ is an open neighborhood of $x_0$ in $B$. Note that under this diffeomorphism, $\mathscr{U}(x_0)\cap M_G$ is identified with $(B^{T_{x_0}M_G}(0,4\e)\times\{0\})$.

In the sequel, we will  indifferently write $B^{T_{x_0}M_G}(0,4\e)\times B^{N_{G,x_0}}(0,4\e)$ or $\mathscr{U}(x_0)$, $x_0$ or 0, etc...

We identify $(L_B)_Z$, $(E_B)_Z$ and $(\E_p)_{B,Z}$ with $(L_B)_{x_0}$, $(E_B)_{x_0}$ and $(\E_p)_{B,x_0}$ by using parallel transport with respect to $\n^{L_B}$, $\n^{E_B}$ and $\n^{(\E_p)_B}$ along the curve $u\in [0,1] \mapsto \gamma_u=uZ$.

Fix $y_0\in \pi^{-1}(x_0)$. We define $\wt{\gamma}\colon [0,1] \to M$ to be the curve lifting $\gamma$ such that $\derpar{\wt{\gamma}_u}{u}\in T^H_{\wt{\gamma}_u}U$. As above, on $\pi^{-1}(B^{T_{x_0}B}(0,4\e))$, we can trivialize $L$, $E$ and $\E_p$ using the parallel transport along $\wt{\gamma}$ with respect to the corresponding connections. By \eqref{def-connection-down}, the previous trivialization are naturally induced by this one.

This also gives a diffeomorphism
\begin{equation}
\pi^{-1}(B^{T_{x_0}B}(0,4\e)) \simeq G\times B^{T_{x_0}B}(0,4\e),
\end{equation}
and the induced $G$-action on $G\times B^{T_{x_0}B}(0,\e)$ is then 
\begin{equation}
\label{action-triviale}
g.(g',Z) = (gg',Z).
\end{equation}

Let $\{e_i^0\}$ and $\{e_i^\perp\}$ be orthonormal basis of $T_{x_0}M_G$ and $N_{G,x_0}$ respectively. Then $\{e_i\}=\{e_i^0,e_i^\perp\}$ is an orthonormal basis of $T_{x_0}B$. Let $\{e^i\}$ be its dual basis. We will also denote $\mathcal{F}_*(e_i^0)$, $\mathcal{F}_*(e_i^\perp)$ by $\{e_i^0\}$, $\{e_i^\perp\}$, so that in our coordinates,
\begin{equation}
\derpar{}{Z_i^0} = e_i^0 \:, \quad \derpar{}{Z_i^\perp} = e_i^\perp.
\end{equation}

In what follows, we will extend the geometric object from $B^{T_{x_0}B}(0,4\e)$ to $\R^{2n-d}\simeq T_{x_0}B$ (here the identification is similar to \eqref{R2n=TM}) to get analogue geometric structures on $G\times \R^{2n-d}$ as on $M$, an thus work on 
\begin{equation}
M_0:=G\times \R^{2n-d}
\end{equation}
 instead of $M$.

Let $L_0$ be the trivial bundle $L|_{G.y_0}$ lifted on $M_0$. We still denote by $\n^L$, $h^L$ the connection and metric on $L_0$ over $\pi^{-1}(B^{T_{x_0}B}(0,4\e))$ induced by the above identification. Then $h^L$ is identified with the constant metric $h^{L_0}=h^{L_{y_0}}$. We use similar notations for the bundle $E$.

Let $\varphi \colon \R \to [0,1]$ be a smooth even function such that
\begin{equation}
\label{def-varphi}
\varphi(v)=\left \{
\begin{aligned}
&1 \text{ for } |v|<2, \\
& 0 \text{ for } |v|>4.
\end{aligned}
\right.
\end{equation}
Let $\varphi_\e\colon M_0\to M_0$ defined by
\begin{equation}
\varphi_\e(g,Z) = (g,\varphi(|Z|/\e)Z).
\end{equation}

Let $\n^{E_0} = \varphi_\e^*\n^E$. Then $\n^{E_0}$ is an extension of $\n^E$ outside $\pi^{-1}(B^{T_{x_0}B}(0,4\e))$.

Let $P^{TY}$ be the orthogonal projection from $TM$ onto $TY$. For $W\in TB$, let $W^H\in T^HU$ be the horizontal lift of $W$. Then we can define the Hermitian connection $\n^{L_0}$   on $(L_0,h^{L_0})\to G\times \R^{2n-d}$   by
\begin{equation}
\label{def-nablaL0}
\n^{L_0} = \varphi_\e^*\n^L + \big(1-\varphi(|Z|/\e)\big)R^L_{y_0}(Z^H,P^{TY}_{y_0}\cdot).
\end{equation}

We can compute directly the curvature $R^{L_0}$ of $\n^{L_0}$: if we denote $(1,Z)$ just by $Z$, then
\begin{equation}
\begin{aligned}
R^{L_0}_Z= &R^L_{\varphi_\e(Z)}(P^{TY}_{y_0}\cdot,P^{TY}_{y_0}\cdot) +R^L_{y_0}(P^{T^HU}_{y_0}\cdot,P^{TY}_{y_0}\cdot)  \\ 
&+ \varphi^2(|Z|/\e) R^L_{\varphi_\e(Z)}(P^{T^HU}_{y_0}\cdot,P^{T^HU}_{y_0}\cdot)\\
&+\varphi(|Z|/\e)\big[ R^L_{\varphi_\e(Z)}-R^L_{y_0} \big](P^{T^HU}_{y_0}\cdot,P^{TY}_{y_0}\cdot)\\
&+\varphi'(|Z|/\e)\frac{Z^*}{\e|Z|}\wedge \big[ R^L_{\varphi_\e(Z)}-R^L_{y_0} \big](Z^H,P^{TY}_{y_0}\cdot)\\
&+(\varphi \varphi')(|Z|/\e)\frac{Z^*}{\e|Z|}\wedge R^L_{\varphi_\e(Z)}(Z^H,P^{T^HU}_{y_0}\cdot),
\end{aligned}
\end{equation}
where $Z^*\in T^*_{x_0}B$ is the dual of $Z\in T_{x_0}B$ with respect to the metric $g^{TB}_{x_0}$.

The group $G$ acts naturally on $M_0$ by \eqref{action-triviale} and under our identifications, the action of $G$ on $L$, $E$ on $G\times ^{T_{x_0}B}(0,\e)$ is exactly the $G$-action on $L|_{G.y_0}$, $E|_{G.y_0}$. 

We define a $G$-action on $L_0$, $E_0$ by the action of $G$ on $G.y_0$. Then it extends the $G$-action on $L$, $E$ on $G\times ^{T_{x_0}B}(0,\e)$ to $M_0$.

By Lemma \ref{TY^perp=JTP}, we know that 
\begin{equation}
\label{RL(1,Z0)}
R^L_{(1,Z^0)}(Z^H,K^M) = R^L_{(1,Z^0)}\big((Z^\perp)^H,K^M\big).
\end{equation}

For $(1,Z)\in G\times \R^{2n-d}$, \eqref{action-triviale} gives $\varphi_{\e*}K^{M_0}_{(1,Z)} = K^M_{y_0}$ for $K\in \g$. Thus, by \eqref{def-mu}, \eqref{def-nablaL0} and \eqref{RL(1,Z0)}, the moment map $\mu_0\colon M_0 \to \g^*$ of the $G$-ac( $(M_0,L_0)$ is given by
\begin{equation}
\label{muM0}
 \mu_0(K)_{(1,Z)} =  \mu(K)_{\varphi_\e(1,Z)}+ \frac{1}{2i\pi} \big(1-\varphi(|Z|/\e)\big)R^L_{y_0}((Z^\perp)^H,K^M_{y_0}).
\end{equation}
Now, from the construction of our coordinate, we have $\mu_0 =0$ on $G\times \R^{2n-2d}\times \{0\}$. Moreover,
\begin{equation}
\label{mu(K)}
\mu(K)_{\varphi_\e(1,Z)} = \frac{1}{2i\pi}R^L_{(1,Z)}\big( \varphi(|Z|/\e)(Z^\perp)^H,K^M\big) +O\big(\varphi(|Z|/\e)|Z||Z^\perp|\big).
\end{equation}
Thus, from our construction, Lemma \ref{bL-ND} and \eqref{id-NG}, \eqref{muM0} and \eqref{mu(K)}, we know that 
\begin{equation}
\mu_0^{-1}(0)=G\times \R^{2n-2d}\times \{0\}.
\end{equation}

Let 
\begin{equation}
\label{gTM0-et-J0}
g^{TM_0}(g,Z)=g^{TM}(\varphi_\e(g,Z)) \quad\text{and} \quad J_0(g,Z)=J(\varphi_\e(g,Z))
\end{equation}
 be the metric and almost-complex structure on $M_0$.  Let $T^{*(0,1)}M_0$ be the anti-holomorphic cotangent bundle of $(M_0,J_0)$. Since $J_0(g,Z)=J(\varphi_\e(g,Z))$, $T^{*(0,1)}_{(g,Z),J_0} M_0$ is naturally identified with $T^{*(0,1)}_{\varphi_\e(g,Z),J}M_0$.

  We can now construct all the objects corresponding to those of  Section \ref{Sect-operator} in this new setting and denotes them by adding subscripts 0, e.g. $\E_{0,p}$, $\n^{\det_0}$, $\n^{\mathrm{Cl}_0}$, $\n^{\mathrm{Bi}_0}$, $\n^{\E_{0,p}}$, ... Then, we can define the Dirac operator $D^{M_0}_p$ on $M_0$, which satisfies
\begin{equation}
\label{def-LM0p}
D_p^{M_0,2} = \Delta ^{\E_{0,p}} -p(2\omega_{d,0} +\tau_0) + \Psi_{\mathcal{E}_0}.
\end{equation}
By \eqref{Lich-D_p^2} and the above constructions, we know that $D_p^2$ and $D_p^{M_0,2}$ coincide on $\pi^{-1}(B^{T_{x_0}B}(0,2\e))$.

 We  can identify $\Wedge(T^*_{(g,Z)}M_0)$ with $\Wedge(T^*_{gy_0}M)$ by identifying first $\Wedge(T^*_{(g,Z)}M_0)$ with $\Wedge(T^*_{\varphi_\e(g,Z),J}M)$ and then identifying $\Wedge(T^*_{\varphi_\e(g,Z),J}M)$ with $\Wedge(T^*_{gy_0}M)$ by parallel transport with respect to $\n^{\mathrm{Bi}_0}$ (see \eqref{def-nablaB}) along $u\in [0,1] \mapsto (g,u\varphi(|Z|/\e)Z)$. We also trivialize $\det(T^{(1,0)}M_0)$ in this way using $\n^{\det_0}$.

Let $g^{TB_0}$ be the metric on $B_0:=\R^{2n-d}$ induced by $g^{TM_0}$, and let $dv_{B_0}$ by the corresponding Riemannian volume. We denote by $TY_0$ the relative tangent bundle of the fibration $M_0\to B_0$, and by $g^{TY_0}$ the metric on $TY_0$ induced by $g^{TM_0}$.

The operator $\Phi D_p^{M_0,2}\Phi^{-1}$ is also well-defined on $T_{x_0}B \simeq \R^{2n-d}$. More precisely, it is an operator on the bundle $(\E_{0,p})_{B_0}$ over $B_0$ induced by $\E_{0,p}$ , and by \eqref{Lich-PhiDp2Phi^-1}, it is given by
\begin{equation}
\label{Lich-PhiL_pM0Phi^-1}
\Phi D_p^{M_0,2} \Phi^{-1} = \Delta ^{(\E_{0,p})_{B_0}} -p(2\omega_{0,d} +\tau_0) + \Psi_{\mathcal{E}_0} -  \langle \wt{\mu}^{\E_{0,p}},\wt{\mu}^{\E_{0,p}} \rangle_{g^{TY_0}} -\frac{1}{h_0} \Delta_{B_0} h_0.
\end{equation}

Let $\exp(-uD_p^{M_0,2})(Z,Z')$  be the smooth heat kernel of $D_p^{M_0,2}$ with respect to $dv_{M_0}(Z')$.
\begin{lemme}
\label{noyau-Lp-et-LpM0}
Under notation of Proposition \ref{the-pb-is-local} and the above trivializations, the following estimate holds uniformly on $v = (g,Z),v' = (g',Z') \in G\times B^{T_{x_0}B}(0,\e)$:
\begin{equation}
\left| e^{-\frac{u}{p}D_p^2}(v,v') -  e^{-\frac{u}{p}D^{M_0,2}_p}\big((g,Z),(g',Z') \big)\right| \leq Cp^N \exp\big(-\frac{\e^2p}{16u}\big).
\end{equation}
\end{lemme}

\begin{proof}
By \eqref{def-LM0p}, $D_p^{M_0,2}$ has the same structure as $D_p^2$. Thus Lemma \ref{estimee-elliptique(p)} and Proposition \ref{the-pb-is-local} are still true if we replace $D_p^2$  therein by $D_p^{M_0,2}$. Moreover, as $D_p^{M_0,2}$ and $D_p^2$ coincide for $|Z|$ small, by the finite propagation speed of the wave equation (see e.g. \cite[Thm. D.2.1]{ma-marinescu}), we know that
\begin{equation}
\F_{\frac{u}{p}} \left( \sqrt{u/p}D_p \right)(v,\cdot) = \F_{\frac{u}{p}} \left( \sqrt{u/p}D_p^{M_0} \right)\big((g,Z),\cdot \big)
\end{equation}
if $v=(g,Z)$ under the above trivializations. Thus, we get our Lemma by \eqref{liensFGH}.
\end{proof}

We still denote  $P_G$  the orthogonal projection from $\Omega^{0,\bullet}(U,L^p\otimes E)$ onto $\Omega^{0,\bullet}(U,L^p\otimes E)^G$. Let $dg$ be the Haar measure on $G$. Then we have 
\begin{equation}
\label{chaleur-sur-invariant}
\big(P_Ge^{-\frac{u}{p}D_p^2}P_G\big)(v,v') = \int_{G\times G} (g,g'^{-1})\cdot e^{-\frac{u}{p}D_p}(g^{-1}v,g'v')dgdg'.
\end{equation}
If we again denote by  $P_G$  the orthogonal projection from $\Omega^{0,\bullet}(M_0,L_0^p\otimes E_0)$ onto $\Omega^{0,\bullet}(M_0,L_0^p\otimes E_0)^G$, then we have a similar formula as \eqref{chaleur-sur-invariant}  for $\big(P_Ge^{-\frac{u}{p}D^{M_0,2}_p}P_G\big)$. Thus, as $G$ preserve every metrics and connections, Lemma \ref{noyau-Lp-et-LpM0} implies
\begin{cor}
 \label{noyau-Lp-et-LpM0-sur-invariant}
Under notation of Proposition \ref{the-pb-is-local}, the following estimate holds uniformly on $v = (g,Z),v' = (g',Z') \in G\times B^{T_{x_0}B}(0,\e)$:
\begin{equation}
\left| \big(P_Ge^{-\frac{u}{p}D_p^2}P_G\big)(v,v) -  \big(P_Ge^{-\frac{u}{p}D_p^{M_0,2}}P_G\big)\big((g,Z),(g',Z') \big)\right| \leq Cp^N \exp\big(-\frac{\e^2p}{16u}\big).
\end{equation}
\end{cor}

Let $S_L$ be a $G$-invariant unit section of $L|_{Gy_0}$. Let $\mathrm{pr}$ be the projection $G\times \R^{2n-d}\to G$. Using $S_L$ and the above discussion, we get two isometries 
\begin{equation}
\E_{0,p} = \Wedge(T^*M_0)\otimes E_0 \otimes L_0^p \simeq \mathrm{pr}^*(\mathcal{E}|_{Gy_0}) \quad \text{and}\quad (\E_{0,p})_{B_0}\simeq \mathcal{E}_{B,x_0}.
\end{equation}
Thus, $\Phi D_p^{M_0,2}\Phi^{-1}$ can be seen as an operator on $\mathcal{E}_{B,x_0}$. Note that our formulas will not depend on the choice of $S_L$ as the isomorphism $\End((\E_{0,p})_{B_0})\simeq \End(\mathcal{E}_{B,x_0})$ is canonical.

Let $dv_{TB}$ be the Riemannian volume of $(T_{x_0}B,g^{TB})$. Recall that $\kappa$ is the smooth positive function defined by
\begin{equation}
\label{def-kappa}
dv_{B_0} (Z) = \kappa(Z) dv_{TB}(Z)=\kappa(Z) dv_{M_G}(x_0)dv_{N_{G,x_0}},
\end{equation}
with $\kappa(0)=1$.

As in \eqref{formule-RFB}, we denote by $R^{L_B}$, $R^{E_B}$ and $R^{\mathrm{Bi}_B}$ the curvature on $L_B$, $E_B$ and $\big(\Wedge(T^*M)\big)_B$ induced by $\n^L$, $\n^E$ and $\n^{\mathrm{Bi}}$ on $M$.

As in \eqref{def-mutildeF}, $\wt{\mu}\in TY$, $\wt{\mu}^E\in TY\otimes \End(E)$ and $\wt{\mu}^{\mathrm{Bi}}\in TY\otimes \End(\Wedge(T^*M))$ are the sections induced by 
$\mu$, $\mu^E$ and $\mu^{\mathrm{Bi}}$ in \eqref{def-muF} and \eqref{relations-mus}.

We denote by $\n_V$ the ordinary differentiation operator on $T_{x_0}B=B_0$ in the direction $V$.

We will now make the change of parameter $t=\frac{1}{\sqrt{p}}\in\, ]0,1]$. 

\begin{defn}
For $s\in \smooth(\R^{2n-d}, \mathcal{E}_{B,x_0})$ and $Z\in\R^{2n-d}$ set
\begin{equation}
\label{def-rescaled}
\begin{aligned}
&(S_ts)(Z)  = s(Z/t), \\
&\n_t = tS_t^{-1} \kappa^{1/2}  \nabla^{(\E_{0,p})_{B_0}}\kappa^{-1/2} S_t, \\
& \n_0= \nabla + \frac{1}{2} R^{L_B}_{x_0}( Z, \cdot) , \\
&\LL_{t}=   t^2 S_t^{-1}  \kappa^{1/2}\Phi D_p^{M_0,2}\Phi^{-1}  \kappa^{-1/2} S_t, \\
&\LL_{0} = - \frac{1}{2}\sum_{i=1}^{2n-d} \left(\n_{0,e_i}\right)^2-2\omega_{d,x_0}-\tau_{x_0} +4\pi^2|P^{TY}\boldrm{J}_{x_0}Z|^2.
\end{aligned}
\end{equation}
\end{defn}	

\begin{prop}
\label{asymp-Lt}
When $t\to 0$, we have
\begin{equation}
\label{asymp-Lt-eq}
\n_{t,e_i} = \n_{0,e_i} + O(t) \text{ and} \quad \LL_{t} = \LL_0 +O(t).
\end{equation}
\end{prop}

\begin{proof}
Let $\Gamma^{L_B}$, $\Gamma^{E_B}$ and $\Gamma^{\mathrm{Bi}_B}$ be the connection form of $\n^{L_B}$, $\n^{E_B}$ and $\n^{\mathrm{Bi}_B}$ with respect to fixed frame of $L_B$, $E_B$ and $\big(\Wedge(T^*M)\big)_B$ which are parallel along the curve $u\in [0,1]\mapsto uZ$ under our trivialization on $B^{T_{x_0}B}(0,4\e)$.

By  \eqref{def-rescaled}, we have for $|Z|\leq \e /t$
\begin{equation}
\label{nablat}
\n_{t,e_i}(Z) = \kappa^{1/2}\left(tZ\right)\left\{ \n_{e_i} + \left( t^{-1}\Gamma^{L_B}_{tZ}(e_i) +t\Gamma^{E_B}_{tZ}(e_i)+t\Gamma^{\mathrm{Bi}_B}_{tZ}(e_i)\right)\right\}\kappa^{-1/2}\left(tZ\right).
\end{equation}

It is a well known fact (see for instance \cite[Lemma 1.2.4]{ma-marinescu}) that for if $\Gamma= \Gamma^{L_B}$ (resp. $\Gamma^{E_B}$, $\Gamma^{\mathrm{Bi}_B}$) and $R=R^L$ (resp. $R^{E_B}$, $R^{\mathrm{Bi}_B}$), then
\begin{equation}
\label{dvlptGamma}
\Gamma_Z(e_i) = \frac{1}{2}R_{x_0}(Z,e_i)+O(|Z|^2).
\end{equation}

Thus,
\begin{equation}
\label{dvlptGammaup}
\begin{aligned}
&t\Gamma^{E_B}_{tZ}(e_i)+t\Gamma^{\mathrm{Bi}_B}_{tZ}(e_i) =    O\left( t^2 \right),  \\
&t^{-1}\Gamma^L_{tZ}(e_i) = \frac{1}{2}R^L_{x_0}(Z,e_i) +O\left( t \right).
\end{aligned}
\end{equation}

The first asymptotic development in Proposition \ref{asymp-Lt} follows from $\varphi(0)=\kappa(0)=1$, \eqref{nablat}, \eqref{dvlptGamma} and \eqref{dvlptGammaup}.  

Let $(g^{ij}(Z))$ is the inverse of the matrix $(g_{ij}(Z)):=(g^{T_{x_0}B}_Z(e_i,e_j))$. By \eqref{def-laplacien-de-Bochner}, \eqref{Lich-PhiL_pM0Phi^-1} and \eqref{def-rescaled} we have
\begin{multline}
\label{formule-Lt}
\LL_t(Z) = -g^{ij}(tZ)  \left(\n_{t,e_i}\n_{t,e_j} -t\n_ {t,\nabla ^{TB_0}_{e_i}e_j} \right)-\langle t\wt{\mu}^{\E_{0,p}},t\wt{\mu}^{\E_{0,p}}\rangle_{g^{TY}}(tZ) \\
-(2\omega_{0,d} +\tau_0)(tZ)  +t^2\Big(\Psi_{\mathcal{E}_0} +\frac{1}{h_0} \Delta_{B_0} h_0\Big)(tZ).
\end{multline}

With the asymptotic of $\n_t$ above,  \eqref{def-laplacien-de-Bochner} and the fact that $g^{ij}(0)=\delta_{ij}$ we find
\begin{equation}
\label{dvpt-Laplacien}
 -g^{ij}(tZ)  \left(\n_{t,e_i}\n_{t,e_j} -t\n_ {t,\nabla ^{TB_0}_{e_i}e_j} \right) =\sum_i \left(\n_{0,e_i}\right)^2+O(t) .
\end{equation}
Moreover,
\begin{equation}
\label{asymp-2e-bout}
-(2\omega_{0,d} +\tau_0)(tZ) + t^2\Big(\Psi_{\mathcal{E}_0} +\frac{1}{h_0} \Delta_{B_0} h_0\Big)(tZ) = -2\omega_{d,x_0}-\tau_{x_0} +O(t).
\end{equation}

Now, by \eqref{mu=moment},  \eqref{def-J} and the fact that $\wt{\mu}_{y_0} = 0$ for $y_0\in P$, $\pi(y_0)=x_0$, we get for $K\in \g$:
\begin{equation}
-\langle \boldrm{J}e_i^H,K^M\rangle_{y_0} = \omega(K^M,e_i^H)_{y_0}=\n_{e_i^H}(\mu(K))(y_0) = \langle \n^{TY}_{e_i^H} \wt{\mu},K^M\rangle_{y_0}.
\end{equation}
Thus,
\begin{equation}
\label{asymp-mutilde}
|\wt{\mu}|^2_{g^{TY}}(Z) = |\n^{TY}_{Z} \wt{\mu}|^2_{g^{TY}} +O(|Z|^3) = |P^{TY}\boldrm{J}_{x_0}Z|^2_{g^{TY}}+O(|Z|^3).
\end{equation}

Note that 
\begin{equation}
\label{norme-tmutilde-Ep}
\langle t\wt{\mu}^{\E_p},t\wt{\mu}^{\E_p}\rangle_{g^{TY}}= -4\pi^2 \frac{1}{t^2}|\wt{\mu}|^2_{g^{TY}}  +\big\langle4i\pi \wt{\mu}+t^2(\wt{\mu}^E+\wt{\mu}^{\mathrm{Bi}}),\wt{\mu}^E+\wt{\mu}^{\mathrm{Bi}}\big\rangle_{g^{TY}}.
\end{equation}
Thus, we get the second asymptotic development in Proposition \ref{asymp-Lt} by using \eqref{formule-Lt}, \eqref{dvpt-Laplacien}, \eqref{asymp-2e-bout}, \eqref{asymp-mutilde} and \eqref{norme-tmutilde-Ep}.
\end{proof}


\subsection{Convergence of the heat kernel}
\label{Sect-convergence}

In this section, we prove the convergence of the heat kernel of the rescaled operator. Note that here we must have a more precise result than in \cite[Sect. 1.6]{ma-marinescu} because in the proof of Theorem \ref{IMI} (see Section \ref{Sect-proof}) we will have to integrate along the normal directions, and thus we need a result of decay in these directions. To obtain it, we draw our inspiration form \cite{ma-zhang}.

Recall that $\mathcal{E}_0 = \Wedge(T^*M_0)\otimes E_0$ and that we have trivialized the Hermitian bundle $(\mathcal{E}_{0,B_0},h^{\mathcal{E}_{0,B_0}})$ on $B_0=T_{x_0}B$ by identifying it to $(\mathcal{E}_{B,x_0},h^{\mathcal{E}_{B,x_0}})$. Recall also that $\mu_0\colon M_0 \to \g^*$ is the moment map of the $G$-action on $M_0$.

Let $\|\cdot\|_{L^2}$ be the $L^2$-norm on $\smooth(B_0,\mathcal{E}_{B,x_0})$ induced by $g^{T_{x_0}B}$ and $h^{\mathcal{E}_{B,x_0}}$ as in \eqref{def-L^2-product}.

Let $\{f_l\}$ be a $G$-invariant orthonormal frame of $TY$ on $\pi^{-1}\big(B^B(x_0,4\e)\big)$, then $\{f_{0,l}(Z)=f_l(\varphi_\e(g,Z))\}$ is a $G$-invariant orthonormal frame of $TY_0$ on $M_0$.

\begin{defn}
 Set
 \begin{equation}
\D_t = \Big\{\n_{t,e_i}\, , \: 1\leq i\leq 2n-d \: ; \: \frac{1}{t} \langle \wt{\mu}_0,f_{0,l}\rangle(tZ)\, , \: 1\leq l\leq d\Big\},
\end{equation}
and for $k\in \N^*$, let $\D^m_t$ be the family of operators $Q$ acting on $\smooth(T_{x_0}B,\mathcal{E}_{B,x_0})$ which can be written in the form $Q=Q_1\dots Q_m$ with $Q_i\in \D_t$.
\end{defn}

For $s \in \smooth(B_0,\mathcal{E}_{B,x_0})$ and $k\in \N^*$, set
\begin{equation}
\label{def-norme(t,k)}
\begin{aligned}
&||s||_{t,0}^2=||s||_{L^2}^2, \\
&||s||_{t,m}^2=||s||_{t,0}^2 + \sum_{\ell=1}^m \sum_{Q\in \D_t^\ell}  ||Qs||_{t,0}^2. \\
\end{aligned}
\end{equation}
We denote by $\sob^m_t$ the Sobolov space $\sob^m(B_0,\mathcal{E}_{B,x_0})$ endowed with the norm $||\cdot ||_{t,m}$, and by $\sob^{-1}_t$ the Sobolev space of order $-1$ endowed with the norm
\begin{equation}
||s||_{t,-1}=\sup_{s'\in\sob^1_p\setminus\{0\}} \frac{\langle s,s'\rangle_{t,0}}{||s'||_{t,0}}\, .
\end{equation}

Finally, if $A\in \LL(\sob^k_t, \sob^m_t)$, we denote by $||A||^{k,m}_t$ the operator norm of $A$ associated with $||\cdot||_{t,k}$ and $||\cdot||_{t,m}$.

Then $\LL_t$ is a formally self-adjoint elliptic operator with respect to $\|\cdot\|_{t,0}$ and is a smooth family of operators with respect to the parameter $x_0\in M_G$.

We denote by $\smooth_c(B_0,\mathcal{E}_{B,x_0})$ the set of smooth section of $\mathcal{E}_{B,x_0}$ over $B_0$ with compact support.

\begin{prop}
\label{estimationLupprop}
There exist constants $C_1,C_2,C_3>0$ such that for any $t\in\, ]0,1]$ and any $s,s'\in\smooth_c(B_0,\mathcal{E}_{B,x_0})$,
\begin{equation}
\label{estimationsLup}
\begin{aligned}
&\langle \LL_t s, sÊ\rangle_{t,0} \geq C_1||s||_{t,1}^2 - C_2||s||_{t,0}^2, \\
& \left|\langle \LL_t s, s'Ê\rangle_{t,0}\right| \leq C_3 ||s||_{t,1} ||s'||_{t,1}. \\
\end{aligned}
\end{equation}
\end{prop}

\begin{proof}
From \eqref{formule-Lt} and \eqref{def-norme(t,k)}, we have
\begin{multline}
\label{<Lts,s>}
\langle \LL_ts,s\rangle_{t,0}=\|\n_ts\|^2_{t,0} -t^2 \big\langle \langle \wt{\mu}^{\E_{0,p}},\wt{\mu}^{\E_{0,p}}\rangle_{g^{TY}}(tZ)s,s\big\rangle_{t,0} \\
+\left\langle S_t^{-1}\Big(-(2\omega_{0,d} +\tau_0) +t^2\Big(\Psi_{\mathcal{E}_0} +\frac{1}{h_0} \Delta_{B_0} h_0\Big)\Big)s,s\right\rangle_{t,0}.
\end{multline}

By \eqref{muM0} and our constructions, we know that for $Z\in T_{\R,x_0}B$ with $|Z|>4\e$,
\begin{equation}
\label{muM0-loin}
\mu^{\E_{0,p}}(K)_{(1,Z)}=2i\pi p \,\mu_0(K)_{(1,Z)} = pR^L_{y_0}\big((Z^\perp)^H,K^X_{y_0}\big).
\end{equation}
Thus, from \eqref{carre-norme-mutilde-F}, \eqref{muM0}, \eqref{norme-tmutilde-Ep} and \eqref{muM0-loin}, we get
\begin{equation}
\label{<Lts,s>-bout-application-moment}
-t^2 \big\langle \langle \wt{\mu}^{\E_{0,p}},\wt{\mu}^{\E_{0,p}}\rangle_{g^{TY}}(tZ)s,s\big\rangle_{t,0} \geq 2\pi^2 \sum_{l=1}^d \Big\| \frac{1}{t} \langle\wt{\mu}_0,f_{0,l}\rangle(tZ)s\Big\|^2_{t,0}-Ct\|s\|^2_{t,0}.
\end{equation}

Now, \eqref{estimationsLup} follows from \eqref{<Lts,s>} and \eqref{<Lts,s>-bout-application-moment}.
\end{proof}

Let $\Gamma$ be the contour in $\C$ defined in Figure \ref{contour-Gamma}.

\begin{figure}[!h]
\hfill
\setlength{\unitlength}{4cm}
\begin{picture}(2.5,1.125)
\put(0.25,0.5){\vector(1,0){1.85}}
\put(0.70,0.025){\vector(0,1){0.95}}
\thicklines \put(0.50,0.25){\line(0,1){0.5}}
\thicklines \put(0.50,0.25){\line(1,0){1,55}}
\thicklines \put(0.50,0,75){\line(1,0){1,55}}
\put(0.20,0.55){$-2C_2$}
\put(0.50,0.5){\circle*{0.029}}
\put(1.25,0.25){${}_\blacktriangleright$}
\put(1.25,0.75){${}_\blacktriangleleft$}
\put(1.25,0.85){$\Gamma$}
\put(0.72,0.78){$i$}
\put(0.72,0.15){$-i$}
\put(0.63,0.4){0}
\end{picture}
\hfill \hfill
\caption{}
\label{contour-Gamma}
\end{figure}

\begin{prop}
\label{estimationsresoprop}
There exist $t_0>0$ and $C>0$, $a,b \in \N$ such that for any $t\in\, ]0,t_0]$ and any $\lambda \in \Gamma$, the resolvant $\left( \lambda-\LL_{t}\right)^{-1}$ exists and
\begin{equation}
\label{estimationsreso}
\begin{aligned}
&\left\| \left( \lambda-\LL_{t}\right)^{-1} \right\|^{0,0}_t \leq C, \\
&\left\| \left( \lambda-\LL_{t}\right)^{-1} \right\|^{-1,1}_t \leq C(1+|\lambda|^2).
\end{aligned}
\end{equation}
\end{prop}

\begin{proof}
Note that $\LL_{t}$ is  self-adjoint operator, thus \eqref{estimationsLup} implies that  $\big( \lambda-\LL_{t}\big)^{-1}$ exists for $\lambda \in \Gamma$ and there is a constant $C>0$ (independent of $\lambda$) such that
\begin{equation}
\left\| \big( \lambda-\LL_{t}\big)^{-1} \right\|^{0,0}_t \leq C.
\end{equation}

On the other hand, if $\lambda_0 \in ]-\infty, -2C_2]$, then \eqref{estimationsLup} also implies that
\begin{equation}
\left\| \big( \lambda_0-\LL_{t}\big)^{-1} \right\|^{-1,1}_t \leq \frac{1}{C_1}.
\end{equation}
Then, using the fact that
\begin{equation}
\label{lambdaetlambda0}
\big( \lambda-\LL_{t}\big)^{-1}=\big( \lambda_0-\LL_{t}\big)^{-1} -(\lambda-\lambda_0)\big( \lambda-\LL_{t}\big)^{-1}\big( \lambda_0-\LL_{t}\big)^{-1},
\end{equation}
we find that
\begin{equation}
\label{norme(-1,0)endegre0}
\left\| \big( \lambda-\LL_{t}\big)^{-1} \right\|^{-1,0}_t \leq \frac{1}{C_1}\left( 1 + C|\lambda-\lambda_0|\right).
\end{equation}
Finally, exchanging the last two factors in \eqref{lambdaetlambda0} and applying \eqref{norme(-1,0)endegre0}, we get
\begin{equation}
\label{estimationsresoendeg0}
\begin{aligned}
\left\| \big( \lambda-\LL_{t}\big)^{-1} \right\|^{-1,1}_t &\leq \frac{1}{C_1} + \frac{|\lambda-\lambda_0|}{C_1^2}\left( 1 + C|\lambda-\lambda_0|\right)  \\
& \leq C(1+|\lambda|^2).
\end{aligned}
\end{equation}
The proof of our Proposition is complete.
\end{proof}

\begin{prop}
\label{estimationcommutateursprop}
Take $m\in \N^*$. Then there exists a contant $C_m>0$ such that for any $t\in\, ]0,1]$,  $Q_1,\dots,Q_m \in \D_t\cup\{Z_i\}_{i=1}^{2n-d}$ and  $s,s'\in \smooth_c(B_0, \mathcal{E}_{B,x_0})$,
\begin{equation}
\label{estimationcommutateurs}
\left| \Big\langle [Q_1,[Q_2,\dots[Q_m,\LL_{t}]\dots]]s,s'\Big\rangle_{t,0} \right| \leq C_m||s||_{t,1}||s'||_{t,1}.
\end{equation}
\end{prop}

\begin{proof}
First, note that $[\n_{t,e_i},Z_j]=\delta_{ij}$. Thus by \eqref{formule-Lt}, we know that $[Z_j,\LL_{t}]$ satisfies \eqref{estimationcommutateurs}.

Using \eqref{muM0} and \eqref{muM0-loin}, we see that $\big(\n_{e_i}\langle\wt{\mu}_0,f_{0,l}\rangle\big)(tZ)$ is uniformly bounded with its derivatives for $t\in[0,1]$, and for $|Z|\geq 4\e$,
\begin{equation}
\label{derivee-loin-mutilde(f0l)}
\big(\n_{e_i}\langle\wt{\mu}_0,f_{0,l}\rangle\big)(Z) = \big(e_i\langle\wt{\mu}_0,f_{0,l}\rangle\big)_{x_0}=\omega_{x_0}(f_{0,l},e_i).
\end{equation}
Thus, $[\frac{1}{t}\langle\wt{\mu}_0,f_{0,l}\rangle(tZ),\LL_t]$ also satisfies \eqref{estimationcommutateurs}.

Let $R^{(L_0)_{B_0}}$ and $R^{(\mathcal{E}_0)_{B_0}}$ be the curvatures of the connections on $(L_0)_{B_0}$ and $(\mathcal{E}_0)_{B_0}$ induced by $\n^{L_0}$, $\n^{E_0}$ and $\n^{\mathrm{Bi_0}}$. Then by \eqref{def-rescaled}, we have
\begin{equation}
\label{comavecnati}
\big[\n_{t,e_i},\n_{t,e_j}\big] = \big( R^{L_{0,B_0}}+ t^2R^{\mathcal{E}_{0,B_0}} \big)_{tZ}(e_i,e_j).
\end{equation}

By \eqref{formule-Lt}, \eqref{derivee-loin-mutilde(f0l)} and \eqref{comavecnati}, we find that $\big[\n_{t,e_i},\LL_{t}\big]$ has the same structure as $\LL_{t}$ for $t\in \, ]0,1]$, by which we mean that it is of the form
\begin{multline}
\label{structure}
\sum_{i,j} a_{ij}(t,tZ) \n_{t,e_i}\n_{t,e_i} + \sum_{i}  b_{i}(t,tZ) \n_{t,e_i} + c(t,tZ) \\
+\sum_l \left[ d_l(t,tZ)\frac{1}{t}\langle\wt{\mu}_0,f_{0,l}\rangle(tZ)+ d'\Big| \frac{1}{t}\wt{\mu}_0\Big|_{g^{TY}}^2 \right],
\end{multline}
where $d'\in \C$, and $a_{ij}$, $b_i$, $c$ and $d_l$ are polynomials in the first variable, and have all their derivatives in the second variable uniformly bounded for $Z\in \R^{2n-d}$ and $t\in [0,1]$. Note that in fact, for $\big[\n_{t,e_i},\LL_{t}\big]$, $d'=0$ in \eqref{structure}.

The adjoint connection $(\n_t)^*$ of $\n_t$ with respect to $\langle \cdot \, , \cdot \rangle_{t,0}$ is given by
\begin{equation}
\label{adjointdenati}
(\n_t)^* = - \n_t -t \big(\kappa^{-1} \n \kappa\big)(tZ).
\end{equation}
Note that the last term of \eqref{adjointdenati} and all its derivative in $Z$ are uniformly bounded for $Z\in \R^{2n-d}$ and $t\in [0,1]$. Thus, by \eqref{structure} and \eqref{adjointdenati}, we find that \eqref{estimationcommutateurs} holds when $m=1$.

Finally, we can prove by induction that $[Q_1,[Q_2,\dots[Q_m,\LL_{t}]\dots]]$ has also the same structure as in \eqref{structure}, and thus satisfies \eqref{estimationcommutateurs} thanks to \eqref{adjointdenati}.
\end{proof}

\begin{prop}
\label{norme(m,m+1)resolvante}
For any $t\in\, ]0,t_0]$, $\lambda \in \Gamma$ and $m\in \N$,
\begin{equation}
(\lambda-\LL_{t})^{-1}\big( \sob^m_t\big) \subset \sob^{m+1}_t.
\end{equation}

Moreover, for any $\alpha \in \N^{2n-d}$, there exist $K\in \N$ and $C_{\alpha,m}>0$ such that for any $t\in \, ]0,1]$, $\lambda \in \Gamma$ and $s\in \smooth_c(B_0, \mathcal{E}_{B,x_0})$,
\begin{equation}
\left\| Z^\alpha (\lambda-\LL_{t})^{-1}s\right\|_{t,m+1} \leq C_{\alpha,m}(1+|\lambda|^2)^K\sum_{\alpha'\leq \alpha} ||Z^{\alpha'}s||_{t,m}.
\end{equation}
\end{prop}

\begin{proof}
Let $Q_1,\dots,Q_m \in \mathcal{D}_t$ and $Q_{m+1},\dots,Q_{m+|\alpha|} \in \left\{ Z_i\right\}_{i=1}^{2n}$. Then we can express the operator $Q_1\dots Q_{m+|\alpha|}(\lambda-\LL_{t})^{-1}$  as a linear combination of operators of the type
\begin{equation}
[Q_1,[Q_2,\dots[Q_\ell,(\lambda-\LL_{t})^{-1}]\dots]] Q_{\ell+1} \dots Q_{m+|\alpha|} \quad \text{with} \quad \ell \leq m+|\alpha|.
\end{equation}

We denote by $\FF_t$ the family of operator $\FF_t = \{ [Q_{j_1},[Q_{j_2},\dots[Q_{j_k},\LL_{t}]\dots]]\}$. Then any commutator $[Q_1,[Q_2,\dots[Q_\ell,(\lambda-\LL_{t})^{-1}]\dots]]$ can be expressed as a linear combination operators of the form
\begin{equation}
\label{forme-des-com}
(\lambda-\LL_{t})^{-1}F_1 (\lambda-\LL_{t})^{-1}F_2 \dots F_{\ell}(\lambda-\LL_{t})^{-1} \quad \text{with} \quad F_j \in  \FF_t.
\end{equation}
Moreover, by Proposition \ref{estimationcommutateursprop}, the norm $\|\cdot\|_t^{1,-1}$ of any element of $\FF_t$ is uniformly bounded by $C$. As a consequence, using Proposition \ref{estimationsresoprop} we see that there is $C>0$ and $N\in \N$ such that the $\|\cdot\|_t^{0,1}$-norm of operators in \eqref{forme-des-com} is bounded by $C(1+|\lambda|^2)^N$. Thus, Proposition \ref{norme(m,m+1)resolvante} holds.
\end{proof}

Let $e^{-\LL_{t}}(Z,Z')$ be the smooth kernel of the operator $e^{-\LL_{t}}$ with respect to $dv_{TB}(Z')$. Let $\pi_{M_G} \colon TB\times_{M_G} TB\to M_G$ be the projection from the fiberwise product $TB\times_{M_G} TB$ onto $M_G$ (here we should rather write $TB|_{M_G}$ but we drop the subscript to simplify the notations). As $\LL_t$ depends on the parameter $x_0\in M_G$, then $e^{-\LL_{t}}(\cdot,\cdot)$ can be viewed as a section of $\pi_{M_G}^*\left(\End(\mathcal{E}_B)\right)$ over $TB\times_{M_G} TB$. 

Let $\n^{\pi_{M_G}^*\End(\mathcal{E}_B)}$  be the  connection on $\pi_{M_G}^*\End(\mathcal{E}_B)$ induced by $\n^{\mathcal{E}_B}$. Then  $\n^{\pi_{M_G}^*\End(\mathcal{E}_B)}$, $h^E$ and $g^{TM}$ induce naturally a $\mathscr{C}^m$-norm for the parameter $x_0\in M_G$ on sections of $\pi_{M_G}^*\left(\End(\mathcal{E}_B)\right)$.

As above, we will decompose any $Z\in T_{x_0}B$ as $Z=Z^0+Z^\perp$, with $Z_0\in T_{x_0}M_G$ and $Z^\perp\in N_{G,x_0}$.

\begin{thm}
\label{estimationderiveesnoyaudeLpu}
There exists $C'>0$ such that for any $m,m',m'',r\in \N$ and $u_0>0$, there is $C>0$ such that for any $t\in\, ]0,t_0]$, $u\geq u_0$ and $Z,Z'\in T_{x_0}B=B_0$ 
\begin{multline}
\label{estimationderiveesnoyaudeLpu:eq}
\sup_{|\alpha|,|\alpha'|\leq m} \big(1+|Z^\perp|+|Z'^\perp|  \big)^{m''} \left| \frac{\partial^{|\alpha|+|\alpha'|}}{\partial {Z}^{\alpha}\partial Z'^{\alpha'}} \derpar{^r}{t^r}e^{-u\LL_{t}}(Z,Z')\right|_{\mathscr{C}^{m'}(M_G)} \\
\leq C\big(1+|Z^0|+|Z'^0|  \big)^{2(n+r+m'+1)+m}\exp\Big(4C_2 u -\frac{C'}{u}|Z-Z'|^2\Big),
\end{multline}
where $|\cdot|_{\mathscr{C}^{m'}(M)}$ denotes the $\mathscr{C}^m$-norm for the parameter $x_0\in M_G$.
\end{thm}

\begin{proof}
By \eqref{estimationsreso}, we know that for $k\in \N^*$,
\begin{equation}
\label{e(-Lup)=integralledecontour}
e^{-u\LL_{t}} = \frac{(-1)^{k-1}(k-1)!}{2i\pi u^{k-1}} \int_\Gamma e^{-u\lambda}(\lambda-\LL_{t})^{-k}d\lambda.
\end{equation}

Then for $m\in \N$, we know from Proposition \ref{norme(m,m+1)resolvante} that for $Q\in \cup_{\ell=1}^m\mathcal{D}_t^\ell$, there are $C_m>0$ and $M\in \N$ such that for $\lambda\in \Gamma$,
\begin{equation}
\label{normeQxresolvante}
\|Q(\lambda-\LL_{t})^{-m}\|_t^{0,0} \leq C_m (1+|\lambda|^2)^M.
\end{equation}
Moreover, taking the adjoint of \eqref{normeQxresolvante}, we deduce
\begin{equation}
\label{normeresolvantexQ}
\|(\lambda-\LL_{t})^{-m}Q\|_t^{0,0} \leq C_m (1+|\lambda|^2)^M.
\end{equation}

From \eqref{e(-Lup)=integralledecontour}, \eqref{normeQxresolvante} and \eqref{normeresolvantexQ}, we have for $Q,Q'\in \cup_{\ell=1}^m\mathcal{D}_t^\ell$:
\begin{equation}
\label{Qe(Lup)Q'}
\left\| Qe^{-u\LL_{t}}Q' \right \|_t^{0,0} \leq C_me^{2C_2u}.
\end{equation}

Let $\|\cdot\|_m$ be the usual Sobolev norm on $\smooth(T_{x_0}B,\mathcal{E}_{x_0})$ induced by $h^{\mathcal{E}_{x_0}}$ and the volume form $dv_{TX}(Z)$:
\begin{equation}
||s||_{m}^2=\sum_{\ell\leq m} \sum_{i_1,\dots,i_\ell}  ||\n_{e_{i_1}}\cdots\n_{e_{i_\ell}}s||_{0}^2.
\end{equation}
Then by \eqref{nablat} and \eqref{def-norme(t,k)},  for any $m\in \N$ there exists $C'_m>0$ such that for $s \in \smooth(T_{x_0}B,\mathcal{E}_{x_0})$ with support in $B^{T_{x_0}B}(0,q)$ and $t\in [0,1]$,
\begin{equation}
\label{equivalenceSob}
\frac{1}{C'_m(1+q)^m} \|s\|_{t,m} \leq \|s\|_m \leq C'_m(1+q)^m \|s\|_{t,m}.
\end{equation}

From \eqref{Qe(Lup)Q'}, \eqref{equivalenceSob} and Sobolev inequalities (for $\|\cdot\|_m$) we find that if $Q,Q'\in \cup_{\ell=1}^m\mathcal{D}_t^\ell$, then
\begin{equation}
\label{estimation-chaleure-sans-Z}
\sup_{|Z|,|Z'|\leq q}  \left| Q_ZQ'_{Z'} e^{-u\LL_{t}}(Z,Z')\right|\leq C(1+q)^{2n+2}e^{2C_2u}.
\end{equation}

Moreover, by Lemma \ref{bL-ND} and \eqref{muM0}, \eqref{mu(K)} and \eqref{muM0-loin}, we have
\begin{equation}
\label{minoration-norme-mutilde0}
\sum_{l=1}^d \Big| \frac{1}{t} \langle \wt{\mu}_0,f_{0,l}\rangle(tZ)\Big|^2 = \Big| \frac{1}{t}\wt{\mu}_0\Big|^2_{g^{TY}}(tZ) \geq C|Z^\perp|^2.
\end{equation}
Thus, \eqref{nablat}, \eqref{estimation-chaleure-sans-Z} and \eqref{minoration-norme-mutilde0} imply \eqref{estimationderiveesnoyaudeLpu:eq} with the exponential $e^{2C_2u}$ for the case where $r=m'=0$ and $C'=0$, i.e., for any $m,m''\in \N$, there is $C>0$ such that for any $t\in\, ]0,t_0]$, $Z,Z'\in T_{x_0}B=B_0$ 
\begin{multline}
\label{cas-r=m'=0-presque}
\sup_{|\alpha|,|\alpha'|\leq m} \big(1+|Z^\perp|+|Z'^\perp|  \big)^{m''} \left| \frac{\partial^{|\alpha|+|\alpha'|}}{\partial {Z}^{\alpha}\partial Z'^{\alpha'}} e^{-u\LL_{t}}(Z,Z')\right| \\
\leq C\big(1+|Z^0|+|Z'^0|  \big)^{2n+2+m}\exp(2C_2u).
\end{multline}

To obtain the right exponential factor in the right hand side of \eqref{estimationderiveesnoyaudeLpu:eq}, we proceed as in the proof of \cite[Thm. 11.14]{MR1316553} (see also \cite[Thm. 4.2.5]{ma-marinescu}).

Recall that the function $f$ is defined in \eqref{def-f}. For $\varsigma>1$ and $a\in \C$, set
\begin{equation}
\label{def-Kvarsigma}
K_{u,\varsigma}(a) = \int_\R  e^{iv\sqrt{2u}a}\exp(-v^2/2)\left(1-f\big(\sqrt{2u}v/\varsigma\big)\right)\frac{dv}{\sqrt{2\pi}}.
\end{equation}
Then there are $C'',C_1>0$ such that for any $c>0$ and $m,m'\in \N$, there is $C>0$ such that for $u\geq u_0$, $\varsigma>1$ and $a\in \C$ with $|\Im(a)|\leq c$, we have
\begin{equation}
\label{borne-Kvarsigma}
|a|^m |K_{u,\varsigma}^{(m')}(a)|\leq C\exp \Big(C''c^2u -\frac{C_1}{u}\varsigma^2\Big).
\end{equation}

For $c>0$, let $V_c$ be the image of $\{a \in \C \: : \: |\Im(a)|\leq c\}$ by the map $a\mapsto a^2$, that is 
\begin{equation}
V_c = \Big\{ \lambda \in \C \: : \: \Re(\lambda)\geq \frac{1}{4c^2} \Im(\lambda)-c^2\Big\}.
\end{equation}
Then the contour $\Gamma$ of Figure \ref{contour-Gamma} satisfies $\Gamma \subset V_c$ for $c$ large enough.

As $K_{u,\varsigma}$ is even, there exist a  unique holomorphic function $\wt{K}_{u,\varsigma}$ such that $\wt{K}_{u,\varsigma}(a^2)=K_{u,\varsigma}(a)$. By \eqref{borne-Kvarsigma}, we have for $\lambda \in V_c$
\begin{equation}
\label{borne-Kvarsigma-tilde}
|\lambda|^m |\wt{K}_{u,\varsigma}^{(m')}(\lambda)|\leq C\exp \Big(C''c^2u -\frac{C_1}{u}\varsigma^2\Big).
\end{equation}

Using the finite propagation speed of the wave equation and \eqref{def-Kvarsigma}, we know that there exists $c'>0$ such that for any $\varsigma>1$
\begin{equation}
\label{Kvarsigma-et-exp}
\wt{K}_{u,\varsigma}(\LL_t)(Z,Z')=e^{-u\LL_t}(Z,Z') \qquad \text{if} \quad |Z-Z'|\geq c'\varsigma.
\end{equation}

From \eqref{borne-Kvarsigma-tilde}, we see that for $k\in \N$, there is a unique holomorphic function $\wt{K}_{u,\varsigma,k}$ defined on a neighborhood of $V_c$ which satisfies the same estimates as $\wt{K}_{u,\varsigma}$ in  \eqref{borne-Kvarsigma-tilde} and
\begin{equation}
\frac{\wt{K}_{u,\varsigma,k}^{(k-1)}(\lambda)}{(k-1)!} = \wt{K}_{u,\varsigma}(\lambda).
\end{equation}
In particular, as in \eqref{e(-Lup)=integralledecontour}, we have
\begin{equation}
\wt{K}_{u,\varsigma}(\LL_t)=\frac{1}{2i\pi}\int_{\Gamma}\wt{K}_{u,\varsigma,k}(\lambda-\LL_t)^{-k}d\lambda.
\end{equation}

Using \eqref{Qe(Lup)Q'} and proceeding as in \eqref{equivalenceSob}-\eqref{cas-r=m'=0-presque}, we find
\begin{multline}
\sup_{|\alpha|,|\alpha'|\leq m} \big(1+|Z^\perp|+|Z'^\perp|  \big)^{m''} \left| \frac{\partial^{|\alpha|+|\alpha'|}}{\partial {Z}^{\alpha}\partial Z'^{\alpha'}} \wt{K}_{u,\varsigma}(Z,Z')\right| \\
\leq C\big(1+|Z^0|+|Z'^0|  \big)^{2n+2+m}\exp\Big(C''c^2 u -\frac{C_1}{u}\varsigma^2\Big).
\end{multline}
For $Z\neq Z'$, we set $\varsigma >1$ such that $\big|\varsigma-\frac{1}{c'}|Z-Z'|\big|<1$ in the previous estimate and get
\begin{multline}
\label{cas-r=m'=0-avec-Kvarsigma}
\sup_{|\alpha|,|\alpha'|\leq m} \big(1+|Z^\perp|+|Z'^\perp|  \big)^{m''} \left| \frac{\partial^{|\alpha|+|\alpha'|}}{\partial {Z}^{\alpha}\partial Z'^{\alpha'}} \wt{K}_{u,\varsigma}(Z,Z')\right| \\
\leq C\big(1+|Z^0|+|Z'^0|  \big)^{2n+2+m}\exp\Big(C''c^2 u -\frac{C_1}{c'^2u}|Z-Z'|^2\Big).
\end{multline}
Now, take $\delta_1=\frac{C''c^2+2C_2}{C''c^2+4C_2}$, then from \eqref{cas-r=m'=0-presque}${}^{\delta_1}\times$\eqref{cas-r=m'=0-avec-Kvarsigma}${}^{1-\delta_1}$ and \eqref{Kvarsigma-et-exp} (and from \eqref{cas-r=m'=0-presque} if $Z=Z'$), we get \eqref{estimationderiveesnoyaudeLpu:eq} for $r=m'=0$, i.e., for all $Z,Z'\in T_{x_0}B$
\begin{multline}
\label{cas-r=m'=0}
\sup_{|\alpha|,|\alpha'|\leq m} \big(1+|Z^\perp|+|Z'^\perp|  \big)^{m''} \left| \frac{\partial^{|\alpha|+|\alpha'|}}{\partial {Z}^{\alpha}\partial Z'^{\alpha'}} e^{-u\LL_{t}}(Z,Z')\right| \\
\leq C\big(1+|Z^0|+|Z'^0|  \big)^{2n+2+m}\exp\Big(4C_2 u -\frac{C'}{u}|Z-Z'|^2\Big).
\end{multline}

We now turn to the case $r\geq 1$. By \eqref{e(-Lup)=integralledecontour}, we have
\begin{equation}
\label{der-e(-Lup)=integralledecontour}
\derpar{^{r}}{t^{r}}e^{-u\LL_{t}} = \frac{(-1)^{k-1}(k-1)!}{2i\pi u^{k-1}} \int_\Gamma e^{-\lambda} \derpar{^{r}}{t^{r}}(\lambda - \LL_{t})^{-1}d\lambda.
\end{equation}
 For $k,q \in \N^*$, set
\begin{equation}
I_{k,r} = \Big\{ (\boldrm{k},\boldrm{r})=(k_i,r_i) \in (\N^*)^{j+1}\times (\N^*)^{j} \: : \: \sum_{i=0}^j k_i = k+j \, , \: \sum_{i=1}^j r_i = r \Big\}.
\end{equation}
For $(\boldrm{k},\boldrm{r}) \in I_{k,r}$, $\lambda \in \Gamma$, $t>0$ set
\begin{equation}
\label{def-Ark(lambda,t,v)}
A_\boldrm{r}^\boldrm{k}(\lambda,t) = (\lambda-\LL_t)^{-k_0}\derpar{^{r_1}\LL_t}{t^{r_1}}(\lambda-\LL_t)^{-k_1} \cdots \derpar{^{r_j}\LL_t}{t^{r_j}}(\lambda-\LL_t)^{-k_j}.
\end{equation}
Then there exist $a_\boldrm{r}^\boldrm{k} \in \R$ such that
\begin{equation}
\label{derivee-resolvante^m-avec-A_r^k}
\derpar{^{r}}{v^{r}}(\lambda-\LL_t)^{-k} = \sum_{(\boldrm{k},\boldrm{r}) \in I_{k,q}} a_\boldrm{r}^\boldrm{k}A_\boldrm{r}^\boldrm{k}(\lambda,t).
\end{equation}

We claim that for any $m\in \N$, $k>2(m+r+1)$ and $Q,Q'\in \cup_{\ell=1}^m\mathcal{D}_t^\ell$, there exist $C>0$, $N\in \N$ such that for $\lambda\in \Gamma$
 \begin{equation}
 \label{QA_k^rQ'-eq}
\big\| QA_\boldrm{r}^\boldrm{k}(\lambda,t) Q's\big\|_{0} \leq C(1+|\lambda|)^N\sum_{|\beta|\leq 2r} \|Z^\beta s\|_0.
\end{equation}

Indeed, we know by \eqref{formule-Lt} that $\derpar{^r}{t^r} \LL_t$ is a combination of
\begin{equation}
\label{termes-possible}
 \Big(\derpar{^{r_1}}{v^{r_1}}g^{ij}(tZ)\Big), \quad \Big( \derpar{^{r_2}}{t^{r_2}}\n_{t,e_i} \Big), \quad \derpar{^{r_1}}{t^{r_1}}\theta(tZ), \quad\derpar{^{r_1}}{t^{r_1}}t\langle \tilde{\mu}^{\E_{0,p}},f_{0,l}(tZ)\rangle,
\end{equation}
where  $\theta$ runs over the functions $r^X$, etc., appearing in \eqref{formule-Lt}.

Now, if $f=g^{ij}$ or $f=\theta$ in  \eqref{termes-possible} (resp. $f=\n_{t,e_i}$ or $f=t\langle \tilde{\mu}^{\E_{0,p}},f_{0,l}(tZ)\rangle$), then for $r_1\geq 1$, $\derpar{^{r_1}}{v^{r_1}}f(tZ)$ is a function of the type $g(tZ)Z^\beta$ where $|\beta|\leq r_1$ (resp. $r_1+1$) and  $g(Z)$ and its derivatives in $Z$ are uniformly bounded for $Z\in \R^{2n}$. 

Let $\mathscr{F}'_{t}$ be the family of operators of the form
\begin{equation}
\mathscr{F}'_{t} = \big\{ [f_{j_1}Q_{j_1},[f_{j_2}Q_{j_2}, \dots [f_{j_m}Q_{j_m},\LL_t]\dots]] \big\},
\end{equation}
where $f_{j_i}$ is smooth and bounded (with its derivatives) and $Q_{j_i} \in \mathcal{D}_t\cup\{Z_l\}_{l=1}^{2n-d}$.

We will now deal with the operator $A_\boldrm{r}^\boldrm{k}(\lambda,t) Q'$. First, we move all the terms $Z^\beta$ in the terms $g(tZ)Z^\beta$ (defined above) to the right-hand side of this operator. To do so, we use the same commutator trick as in the proof of Theorem \ref{norme(m,m+1)resolvante}, that is we perform the commutations once at a time with each $Z_i$ (and not directly with $Z^\beta$, $|\beta|>1$). Then we obtain that $A_\boldrm{r}^\boldrm{k}(\lambda,t) Q'$ is of the form $ \sum_{|\beta|\leq 2r}L_{t,\beta}Q''_\beta Z^\beta$ where $Q''_\beta$ is obtained from $Q'$ and its commutation with $Z^\beta$. Next, we move all the terms $\n_{t,e_i}$ and $\langle \frac{1}{t}\tilde{\mu}^{\E_{0,p}},f_{0,l}(tZ)\rangle$  in $\derpar{^r}{t^r} \LL_t$ to the right-hand side of the operators $L_{t,\beta}$. Then as in the proof of  Theorem  \ref{norme(m,m+1)resolvante}, we finally get that $QA_\boldrm{r}^\boldrm{k}(\lambda,t) Q'$ is of the form $ \sum_{|\beta|\leq 2r}\LL_{t,\beta} Z^\beta$, where $\LL_{t,\beta}$ is a linear combination of operators of the type
\begin{equation}
\label{model-L_t,beta^v}
Q(\lambda-\LL_t)^{-k'_0}R_1 (\lambda-\LL_t)^{-k'_1} R_2 \cdots R_{l'}(\lambda-\LL_t)^{-k'_{l'}}Q'''Q'',
\end{equation}
where $\sum_j k'_j = k+l'$, $R_j \in \mathscr{F}_{t}'$, $Q'''\in \cup_{\ell=1}^{2r}\mathcal{D}_t^\ell$ and $Q''\in\cup_{\ell=1}^m\mathcal{D}_t^\ell$ is obtained from $Q'$ and its commutation with $Z^\beta$. Since $k>2(m+r+1)$, we can use Proposition \ref{norme(m,m+1)resolvante} and the arguments leading to \eqref{normeQxresolvante} and \eqref{normeresolvantexQ} in order to split the operator in \eqref{model-L_t,beta^v} into two parts:
\begin{multline}
Q(\lambda-\LL_t)^{-k'_0}R_1 (\lambda-\LL_t)^{-k'_1} R_2 \cdots R_{i}(\lambda-\LL_t)^{-k''_{i}} \times \\
(\lambda-\LL_t)^{-(k'_i-k''_{i})}R_{i+1} \cdots R_{l'}(\lambda-\LL_t)^{-k'_{l'}}Q'''Q'',
\end{multline}
such that the $\|\cdot\|_t^{0,0}$-norm each part is bounded by $C(1+|\lambda|^2)^N$. This conclude the proof of  \eqref{QA_k^rQ'-eq}.

By \eqref{der-e(-Lup)=integralledecontour}, \eqref{derivee-resolvante^m-avec-A_r^k} and \eqref{QA_k^rQ'-eq}, we get \eqref{estimationderiveesnoyaudeLpu:eq} for $m'=0$
using a similar reasoning that for \eqref{cas-r=m'=0}.

For $m'=1$, observe that if $U \in T M_G$, then
\begin{equation}
\label{nablae(-Lup)=integralledecontour}
\nabla^{\pi_{M_G}^*\End(\mathcal{E})}_Ue^{-u\LL_{p}} = \frac{(-1)^{k-1}(k-1)!}{2i\pi u^{k-1}} \int_\Gamma e^{-\lambda}\nabla^{\pi_{M_G}^*\End(\mathcal{E})}_U(\lambda-\LL_{t})^{-k}d\lambda.
\end{equation}
Moreover, $\nabla^{\pi_{M_G}^*\End(\mathcal{E})}_U(\lambda-\LL_{t})^{-k}$ is a linear combination operators of the form
\begin{equation}
\label{decompo:nabla(lambda-Lup)^(-k)}
(\lambda-\LL_{t})^{-i_1}\big(\nabla^{\pi_{M_G}^*\End(\mathcal{E})}_U\LL_{t}\big) (\lambda-\LL_{t})^{-i_2} \dots \big(\nabla^{\pi_{M_G}^*\End(\mathcal{E})}_U\LL_{t}\big)(\lambda-\LL_{t})^{-i_\ell},
\end{equation}
and $\nabla^{\pi_{M_G}^*\End(\mathcal{E})}_U\LL_{t}$ is a differential operator with the same structure as $\LL_{t}$. In particular, $\nabla^{\pi_{M_G}^*\End(\mathcal{E})}_U\LL_{t}$ satisfies an estimates analogous to \eqref{estimationcommutateurs}. Thus, above arguments can be repeated to prove \eqref{estimationderiveesnoyaudeLpu:eq} for $m'=1$. The case $m'\geq 2$ is similar.
\end{proof}

\begin{rem}
 In the sequel, we will in fact only use Theorem \ref{estimationderiveesnoyaudeLpu} with $r=0,1$, but we prefer to state it in the general case.
\end{rem}

\begin{prop}
\label{diffdesresolvantes-prop}
There are constants $C>0$ and $M\in \N^*$ such that for $t\in [0,t_0]$ and $\lambda\in \Gamma$,
\begin{equation}
\label{diffdesresolvantes}
\left\| \big((\lambda-\LL_{t})^{-1} - (\lambda-\LL_{0}
)^{-1}\big)s \right\|_{0,0} \leq Ct (1+|\lambda|^2)^M\sum_{|\alpha|\leq3}||Z^\alpha s||_{0,0}.
\end{equation}
\end{prop}

\begin{proof}
 From \eqref{nablat} and \eqref{def-norme(t,k)}, for $t\in [0,1]$ and $m\in \N^*$ we find
\begin{equation}
\|s\|_{t,m} \leq C\sum_{|\alpha|\leq m} \|Z^\alpha s\|_{0,m}.
\end{equation}
Moreover, for $s,s'$ with compact support, a Taylor expansion of \eqref{formule-Lt} gives
\begin{equation}
\left | \left\langle (\LL_{t}-\LL_{0}
)s, s' \right\rangle_{t,0} \right | \leq Ct \|s'\|_{t,1} \sum_{|\alpha|\leq 3} \|Z^\alpha s \|_{0,1}.
\end{equation}
Thus,
\begin{equation}
\label{Lup-Luinfty}
\| (\LL_{t}-\LL_{0})s \|_{t,-1} \leq Ct\sum_{|\alpha|\leq 3} \|Z^\alpha s \|_{0,1}.
\end{equation}

Note that
\begin{equation}
\label{Lup-Luinfty,resolvantes}
(\lambda-\LL_{t})^{-1}-(\lambda-\LL_{0}
)^{-1} = (\lambda-\LL_{t})^{-1}(\LL_{t}-\LL_{0}
)(\lambda-\LL_{0}
)^{-1}.
\end{equation}
Moreover, Propositions \ref{estimationsresoprop}, \ref{estimationcommutateursprop} and \ref{norme(m,m+1)resolvante} still holds for $t=0$. Thus, Proposition \ref{norme(m,m+1)resolvante}, \eqref{Lup-Luinfty} and \eqref{Lup-Luinfty,resolvantes} yields to \eqref{diffdesresolvantes}.
\end{proof}

\begin{thm}
\label{cvcenoyauLup->noyauLuinfini}
There exists $C'>0$ such that for any $m,m',m''\in \N$ and $u_0>0$, there is $C>0$ such that for any $t\in\, ]0,t_0]$, $u\geq u_0$ and $Z,Z'\in B_0$ 
\begin{multline}
\label{cvcenoyauLup->noyauLuinfini:eq}
\sup_{|\alpha|,|\alpha'|\leq m} \big(1+|Z^\perp|+|Z'^\perp|  \big)^{m''} \left| \frac{\partial^{|\alpha|+|\alpha'|}}{\partial {Z}^{\alpha}\partial Z'^{\alpha'}} \left(e^{-u\LL_{t}}-e^{-u\LL_{0}
}\right)(Z,Z')\right|_{\mathscr{C}^{m'}(M_G)} \\
\leq Ct\big(1+|Z^0|+|Z'^0|  \big)^{2(n+m'+1)+m} \exp\Big(4C_2 u -\frac{C'}{u}|Z-Z'|^2\Big).
\end{multline}
\end{thm}

\begin{proof}
Let $\mathcal{B}_q=B^{T_{x_0}B}(0,q)$. Let $\|s\|^2_{\mathcal{B}_q}=\int_{|Z|\leq q} |s|_{h^{\mathcal{E}_{x_0}}}^2dv_{TX}(Z)$, and let $J_{q,x_0} = L^2(\mathcal{B}_q,\mathcal{E}_{B,x_0})$. If $A$ is a bounded operator on $J_{q,x_0}$, we denote its operator norm by $\|A\|_{\mathcal{B}_q}$. By \eqref{e(-Lup)=integralledecontour} and \eqref{diffdesresolvantes}, we know that there is  $C'>0$ and $N,M\in \N$ such that for $t\in \, ]0,1]$, 
\begin{equation}
\label{norme_B}
\begin{aligned}
\left\| e^{-u\LL_{t}}-e^{-u\LL_{0}}\right\|_{\mathcal{B}_q} &\leq \frac{1}{2\pi} \int_\Gamma |e^{-u\lambda}| \left\| (\lambda-\LL_{t})^{-1}-(\lambda-\LL_{0}
)^{-1} \right\|_{\mathcal{B}_q}d\lambda \\
& \leq Ct \int_\Gamma e^{-u\mathrm{Re}(\lambda)} (1+|\lambda|^2)^M(1+q)^N d\lambda \leq C't(1+q)^N.
\end{aligned}
\end{equation}

Let $\phi \colon T_{x_0}B\to [0,1]$ be a smooth function with compact support, equal to 1 near 0 and such that $\int_{T_{x_0}B}\phi(Z) dv_{TX}(Z)=1$. Let $\nu\in ]0,1]$. By the proof of Theorem \ref{estimationderiveesnoyaudeLpu}, we see that $e^{-u\LL_{0}
}$ satisfies an  inequality similar to \eqref{estimationderiveesnoyaudeLpu:eq}. By Theorem \ref{estimationderiveesnoyaudeLpu}, there exists $C>0$ such that for $|Z|,|Z'|\leq q$ and $U,U'\in \mathcal{E}_{x_0}$, 
\begin{multline}
\left | \left \langle \big(e^{-u\LL_{t}}-e^{-u\LL_{0}
}\big)(Z,Z') U, U' \right \rangle \right. \\
-\int_{T_{x_0}B\times T_{x_0}B}  \left \langle \big(e^{-u\LL_{t}}-e^{-u\LL_{0}
}\big)(Z-W,Z'-W') U, U' \right \rangle \\
\left. \times \frac{1}{\nu^{4n-2d}} \phi(W/\nu)\phi(W'/\nu) dv_{TX}(W)dv_{TX}(W') \right| \leq C\nu (1+q)^N |U||U'|.
\end{multline}

Moreover, by \eqref{norme_B}, we have
\begin{multline}
\left | \int_{T_{x_0}B\times T_{x_0}B}  \left \langle \big(e^{-u\LL_{t}}-e^{-u\LL_{0}
}\big)(Z-W,Z'-W') U, U' \right \rangle \right. \\
\left. \times \frac{1}{\nu^{4n-2d}} \phi(W/\nu)\phi(W'/\nu) dv_{TX}(W)dv_{TX}(W') \right| \leq \frac{Ct}{\nu^{2n-d}}(1+q)^N |U||U'|.
\end{multline}
Hence, taking $\nu=t^{1/(2n-d+1)}$ we find that there is $C>0$ and $K\in \N$ such that for any $t\in\, ]0,t_0]$, $Z,Z'\in B^{B_0}(0,q)$, 
\begin{equation}
 \Big|  \big(e^{-u\LL_{t}}-e^{-u\LL_{0}}\big)(Z,Z')\Big| \leq Ct^{1/(2n-d+1)}(1+q )^{K}.
\end{equation}
In particular, we have
\begin{equation}
\label{t=0-et-0}
e^{-u\LL_{t}}\big|_{t=0} = e^{-u\LL_{0}}.
\end{equation}

From Theorem \ref{estimationderiveesnoyaudeLpu}, \eqref{t=0-et-0} and the  formula
\begin{equation}
G(t)-G(0)=\int_0^t G'(s)ds,
\end{equation}
we get \eqref{cvcenoyauLup->noyauLuinfini:eq}.
\end{proof}

\begin{rem}
 As we have estimates on every derivatives of $e^{-u\LL_{t}}(Z,Z')$, we can in fact use the same method as in Theorem \ref{cvcenoyauLup->noyauLuinfini} to get an asymptotic expansion at every order of $e^{-u\LL_{t}}(Z,Z')$.
\end{rem}

\subsection{Computation of the limiting heat kernel}
\label{Sect-calcul}

 In this section, we will evaluate the limiting heat kernel $e^{-u\LL_{0}}((0,Z^\perp),(0,Z^\perp))$ for $(0,Z^\perp)\in T_{x_0}B$ and thus obtain Theorem \ref{thm-limit-near}.

Recall that we have the following splitting of vector bundle over $P$, which is orthogonal for both $b^L$ and $g^{TM}$ (see \eqref{g^TM|_P} and \eqref{decompo-TU|P}):
\begin{equation}
TU = T^HP \oplus TY \oplus JTY.
\end{equation}
Note also that by \eqref{def-b^L} and \eqref{def-J}, we have
\begin{equation}
b^L(\cdot \,, \cdot) = \langle(-J\boldrm{J})\cdot \,, \cdot\rangle,
\end{equation}
and thus $-J\boldrm{J}$ preserves both $TY$ and $JTY$ on $P$. In particular, on $P$,  $\boldrm{J}$ intertwines $TY$ and $JTY$, and is invertible on $TY\oplus JTY$ because $g^{TM}$ and $b^L$ are definite positive on this bundle. Thus,
\begin{equation}
\label{J2TY=TY}
\boldrm{J}^2TY=TY\, , \quad \boldrm{J}TY=JTY \,,\quad \boldrm{J}T^HP =JT^HP= T^HP.
\end{equation}
Thus, $\boldrm{J}$ induces naturally $\boldrm{J}_G\in \End(TM_G)$, and  we see with \eqref{id-NG} that $(\boldrm{J}TY)_B|_{M_G}$ is the orthogonal complement of $TM_G$ in $TB$. We will identify the normal bundle $N_G$ of $M_G$ in $B$ with $(\boldrm{J}TY)_B|_{M_G}$. From this fact and \eqref{J2TY=TY}, we know that for $U,V\in T_{x_0}B$,
\begin{equation}
\label{omega-et-PTMG}
\omega(U^H,V^H)=\omega_G(P^{TM_G}U,P^{TM_G}V).
\end{equation}

From the above discussion, we can diagonalize $\boldrm{J}$ on $(T^HP)^{(1,0)}$ and $(TY\oplus JTY)^{(1,0)}$, and we thus can get orthonormal basis $\{w_j^0\}_{j=1}^{n-d}$ and $\{e_i^\perp\}_{i=1}^d$  of $T^{(1,0)}_{x_0}M_G$ and $N_{G,x_0}=(\boldrm{J}TY)_{B,x_0}\subset TB$ respectively such that in these basis
\begin{equation}
\label{def-aj}
\left\{
\begin{aligned}
&\boldrm{J}|_{T^{(1,0)}_{x_0}M_G} = \frac{\ic}{2\pi} \mathrm{diag}(a_1^0,\dots,a_{n-d}^0), \\
&\boldrm{J}^2|_{N_{G,x_0}} = -\frac{1}{4\pi^2} \mathrm{diag}(a_1^{\perp,2},\dots,a_d^{\perp,2}),
\end{aligned}
\right.
\end{equation}
where $a_j^0\, \in \R$ and $a_j^\perp\in \R^*$ are the respective eigenvalues of $-2\ic\pi \boldrm{J}|_{(T^HP)^{(1,0)}}$ and $-2\ic\pi \boldrm{J}|_{(TY\oplus JTY)^{(1,0)}}$. Let $\{w^{0,j}\}_{j=1}^{n-d}$ and $\{e^{\perp,i}\}_{i=1}^d$ be their dual basis. We also set
\begin{equation}
 e^0_{2j-1}=\frac{1}{\sqrt{2}}(w_j^0+\bw_j^0) \quad \text{ and } \quad e^0_{2j}=\frac{\ic}{\sqrt{2}}(w_j^0-\bw_j^0).
\end{equation}
Then $\{e_i^0\}_{i=1}^{2n-2d}$ is an orthonormal basis of $T_{x_0}M_G$.

From now on, we will use the coordinates in Section \ref{Sect-rescaling} induced by the above basis as in \eqref{R2n=TM}.

We denote by $Z^0=(Z_1^0,\dots,Z_{2n-2d}^0)$ and $Z^\perp=(Z_1^\perp,\dots, Z_d^\perp)$ the elements in $T_{x_0}M_G$ and $N_{G,x_0}$. Then $Z\in T_{x_0}B$ can be decomposed as $Z=(Z^0,Z^\perp)$. We will also use the complex coordinates $z^0=(z_1^0,\dots,z_{n-d}^0)$, so that
\begin{equation}
\begin{aligned}
&Z^0=z^0+\ol{z}^0, \\
&w_j^0 = \sqrt{2}\derpar{}{z_j^0}\: , \quad \bw_j^0 = \sqrt{2}\derpar{}{\ol{z}_j^0},\\
&e_{2j-1}^0 = \derpar{}{z_j^0}+\derpar{}{\ol{z}_j^0} \: , \quad e_{2j}^0= \ic(\derpar{}{z_j^0}-\derpar{}{\ol{z}_j^0}).
\end{aligned}
\end{equation}
When we consider $z^0$ or $\ol{z}^0$ as vector fields, we identify them with $\sum_j z_j^0\derpar{}{z_j^0}$ and $\sum_j \ol{z}_j^0\derpar{}{\ol{z}_j^0}$. Note that 
\begin{equation}
\Big| \derpar{}{z_j^0}\Big|^2=\Big| \derpar{}{\ol{z}_j^0}\Big|^2=\frac{1}{2} \quad \text{and} \quad |z^0|^2=|\ol{z}^0|^2=\frac{1}{2}|Z^0|^2.
\end{equation}

Set
\begin{equation}
\label{def-L}
\LL = -\sum_{i=1}^{2n-2d}(\n_{0,e_i^0})^2 - \sum_{j=1}^{n-d}a_j^0\,,
\end{equation}
and recall that
\begin{equation}
\LL^\perp = -\sum_{i=1}^{d}\left((\n_{e_i^\perp})^2-|a_i^\perp Z_i^\perp|^2\right) - \sum_{j=1}^{d}a_j^\perp.
\end{equation}
As in \cite[(3.11) and (3.13)]{ma-zhang}, we can show using \eqref{courbure-sur-MG}, \eqref{def-rescaled} and \eqref{omega-et-PTMG},  that 
\begin{equation}
\label{L0=L+Lperp+omegad}
\begin{aligned}
&R^{L_B}_{x_0}(U,V) = -2\pi \ic \langle \boldrm{J}P^{TM_G}U,P^{TM_G}V\rangle, \\
&\LL_0 = \LL + \LL^\perp - 2\omega_{d}(x_0),
\end{aligned}
\end{equation}

Thus, 
\begin{equation}
\label{chaleur-L0-L-et-Lperp}
e^{-u\LL_0}(Z,Z') = e^{-u\LL}(Z^0,Z'^0)e^{-u\LL^\perp}(Z^\perp,Z'^\perp)e^{2u\omega_{d}(x_0)}.
\end{equation}
Moreover, using \eqref{def-aj}, \eqref{def-L}, \eqref{L0=L+Lperp+omegad} and the formula for the heat kernel of a harmonic oscillator (see \cite[(E.2.4), (E.2.5)]{ma-marinescu} for instance), we find (with the convention of Theorem \ref{thm-limit-near}):
\begin{equation}
\label{chaleur-L}
e^{-u\LL}(0,0)=\frac{1}{(2\pi)^{n-d}}\frac{\det(\dot{R}_x^{L_G})}{\det\big(1-\exp(-2u\dot{R}_x^{L_G})\big)}.
\end{equation}

We can now prove Theorem \ref{thm-limit-near}. We fix $u>0$.

Let $s\in \smooth_c(B_0,\mathcal{E}_{x_0})$. Then by \eqref{def-kappa} and \eqref{def-rescaled}
\begin{align}
e^{-u\LL_{t}}s(Z) &=S_t^{-1}\kappa^{1/2}e^{-\frac{u}{p}\Phi D_p^{M_0,2}\Phi^{-1}}\kappa^{-1/2}S_t(Z) \notag \\
&=   \kappa(tZ) \int_{\R^{2n-d}}e^{-\frac{u}{p}\Phi D_p^{M_0,2}\Phi^{-1}}(tZ,Z')(S_ts)(Z')\kappa^{1/2}(Z') dv_{TX}(Z')  \\
&=p^{-n+d/2} \kappa(tZ) \int_{\R^{2n-d}}e^{-\frac{u}{p}\Phi D_p^{M_0,2}\Phi^{-1}}(tZ,tZ'')s(Z'')\kappa^{1/2}(tZ'') dv_{TX}(Z''),\notag
\end{align}
which yields to
\begin{equation}
\label{noyauxdeMupetdeLup}
e^{-u\LL_{t}}(Z,Z')=p^{-n+d/2} e^{-\frac{u}{p}\Phi D_p^{M_0,2}\Phi^{-1}}(tZ,tZ')\kappa^{1/2}(tZ)\kappa^{-1/2}(tZ').
\end{equation}

On the other hand, for  $s\in \smooth_c(B_0,(\E_{0,p})_{B_0})$ and $v\in M_0$,
\begin{equation}
\begin{aligned}
\left(e^{-\frac{u}{p}\Phi D_p^{M_0,2}\Phi^{-1}}s\right)(\pi(v)) &= \left(\Phi e^{-\frac{u}{p} D_p^{M_0,2}}\Phi^{-1}s\right)(\pi(v))\\
&=h(v)\int_{M_0} e^{-\frac{u}{p} D_p^{M_0,2}}(v,v') h^{-1}(v')s(v')dv_{M_0}(v') \\
&=h(v) \int_{B_0} e^{-\frac{u}{p} D_p^{M_0,2}}(v,y') h(y')s(y')dv_{B_0}(y'),
\end{aligned}
\end{equation}
thus we find
\begin{equation}
\label{noyau-et-Phi}
h(v)h(v')\big(P_Ge^{-\frac{u}{p}D_p^{M_0,2}}P_G\big)(v,v') = e^{-\frac{u}{p}\Phi D_p^{M_0,2}\Phi^{-1}}(\pi(v),\pi(v')).
\end{equation}

Let $v = (g,Z) \in U\simeq G\times B^{T_{x_0}B}(0,\e)$. We suppose that in the decomposition $Z=Z^0+Z^\perp$, we have $Z^0=0$. Then from Corollary \ref{noyau-Lp-et-LpM0-sur-invariant}, Theorem \ref{cvcenoyauLup->noyauLuinfini}, \eqref{noyauxdeMupetdeLup}, and  \eqref{noyau-et-Phi}, we find that for any $m,m'\in \N$, there exists $C>0$ (independent of $Z^\perp$) such 
\begin{multline}
\label{chaleur-Lp-et-L0}
 \Big| p^{-n+d/2}h(v)h(v)\big(P_Ge^{-\frac{u}{p}D_p^2}P_G\big)(v,v) - \kappa^{-1}(Z^\perp)e^{-u\LL_0}(\sqrt{p}Z^\perp,\sqrt{p}Z^\perp) \Big|_{\mathscr{C}^{m'}(M_G)} \\
\leq Cp^{-1/2}  \big(1+\sqrt{p}|Z^\perp|\big)^{-m}.
\end{multline}

Now, for $v\in U$, we write  as in the Introduction of this paper $v=(y,Z^\perp)$ with $y\in P$ and $Z^\perp\in N_{P/U,y}$. Let $x=\pi(y)\in M_G$. Then we do the procedure of Sections \ref{Sect-rescaling} and \ref{Sect-calcul} with $x_0=x$ and $y_0=y$. Then Theorem \ref{thm-limit-near} follows from \eqref{chaleur-L0-L-et-Lperp}, \eqref{chaleur-L} and \eqref{chaleur-Lp-et-L0} applied to $Z=(0,Z^\perp)\in T_{x_0}B=T_{x_0}M_G\oplus N_{G,x_0}$.

\section{Proof of the inequalities}
\label{Sect-proof-IMI}

In this Section, we prove our main results: Theorems \ref{IMI} and \ref{IMI-general}. In Section \ref{Sect-proof} we prove Theorem \ref{Euler-et-chaleur} and, as a consequence, we obtain the $G$-invariant holomorphic Morse inequalities in the case of a free $G$-action on $P$. Then, we explain in Section \ref{Sect-action-loc-libre} how to modify the arguments in Sections  \ref{Sect-AD-heat-kernel} and \ref{Sect-proof} to get our inequalities under Assumption \ref{assum-regular-value} in full generality. Finally, in Section \ref{autres-composantes}, we apply Theorem \ref{IMI-general} to get estimates on the other isotypic components of the cohomology $H^\bullet(M,L^p\otimes E)$.

\subsection{Proof of Theorem \ref{IMI} when $G$ acts freely on $P$}
\label{Sect-proof}

We assume in this Section that $G$ acts freely on $P$ and $\ol{U}$. We keep here the notations of Sections \ref{Sect-AD-heat-kernel}.

In this section, we will first prove Theorem  \ref{Euler-et-chaleur}, and then show how to use it in conjunction with the convergence of the heat kernel of the rescaled operator to get Theorem \ref{IMI}. The method is inspired by \cite{MR886814} (see also \cite[Sect. 1.7]{ma-marinescu}).

For $0\leq q \leq n$, set
\begin{equation}
b_q^{p,G} = \dim H^q(M,L^p\otimes E)^G.
\end{equation}
By Hodge theory, there is a $G$-equivariant isomorphism $H^\bullet (M,L^p \otimes E) \simeq \ker D_p^2$, and in particular we get for the invariant part:
\begin{equation}
\label{Hodge}
H^\bullet (M,L^p \otimes E)^G \simeq (\ker D_p^2)^G \quad \text{and }\quad b_q^{p,G} = \dim (\ker D_p^2)^G.
\end{equation}

 We begin by proving Theorem  \ref{Euler-et-chaleur}.

\begin{proof}[Proof of Theorem  \ref{Euler-et-chaleur}]
If $\lambda$ is an eigenvalue of $D_p^2$ acting on $\Omega^{0,j}(M,L^p\otimes E)^G$, we denote by $F_j^\lambda$ the corresponding finite-dimensional eigenspace. As $\db^{L^p\otimes E}$ and $\db^{L^p\otimes E,*}$ act on $\Omega^{0,j}(M,L^p\otimes E)^G$ and commute with $D_p^2$, we deduce that 
\begin{equation}
\db^{L^p\otimes E}(F^\lambda_j)\subset F^\lambda_{j+1} \qquad \text{and} \qquad  \db^{L^p\otimes E,*}(F^\lambda_j)\subset F^\lambda_{j-1}.
\end{equation}
 
As a consequence, we have a complexe
\begin{equation}
\label{complexe-espace-propre}
0\longrightarrow F^\lambda_0 \overset{\db^{L^p\otimes E}}{\longrightarrow}F^\lambda_1 \overset{\db^{L^p\otimes E}}{\longrightarrow} \cdots \overset{\db^{L^p\otimes E}}{\longrightarrow}F^\lambda_n \longrightarrow 0.
\end{equation}
If $\lambda=0$, we have  $F^0_j \simeq  H^j(M,L^p\otimes E)^G$ by \eqref{Hodge}. If $\lambda>0$, then the complex \eqref{complexe-espace-propre} is exact. Indeed, if $\db^{L^p\otimes E}s=0$  and $s\in F^\lambda_j$, then
\begin{equation}
s=\lambda^{-1}D_p^2s = \lambda^{-1}\db^{L^p\otimes E}\db^{L^p\otimes E,*}s \in \Im(\db^{L^p\otimes E}).
\end{equation}
In particular, we get for $\lambda>0$
\begin{equation}
\label{sum-dim(Flambdaj)}
\sum_{j=0}^q(-1)^{q-j}\dim \, F^\lambda_j = \dim \big( \db^{L^p\otimes E}(F^\lambda_q) \big)\geq 0,
\end{equation}
with equality if $q=n$. 

Now, 
\begin{equation}
\label{trace-et-vap}
\tr_j[P_Ge^{-\frac{u}{p}D_p^2}P_G]=b_j^{p,G} +\sum_{\lambda>0}e^{-\frac{u}{p}\lambda}\dim\, F^\lambda_j.
\end{equation}
Thus, \eqref{sum-dim(Flambdaj)} and \eqref{trace-et-vap} entail \eqref{Euler-et-chaleur-eq}.

Note that this proof does not depend on the metric we chose on $TM$, so we get \eqref{Euler-et-chaleur-eq} in general.
\end{proof}

We denote by $\tr_{\Lambda^{0,q}}$ the trace on  $\Lambda^{0,q}(T^*M)\otimes L^p \otimes E$ or $\Lambda^{0,q}(T^*M)$. We know that
\begin{equation}
\tr_q[P_Ge^{-\frac{u}{p}D_p^2}P_G]=\int_M \tr_{\Lambda^{0,q}}\left[\big(P_Ge^{-\frac{u}{p}D_p^2}P_G\big)(v,v)\right]dv_M(v).
\end{equation}
With Theorem \ref{thm-estimate-away} and \eqref{chaleur-sur-invariant}, we in fact have
\begin{equation}
\label{trace=trace-sur-U}
\tr_q[P_Ge^{-\frac{u}{p}D_p^2}P_G]=\int_U \tr_{\Lambda^{0,q}}\left[\big(P_Ge^{-\frac{u}{p}D_p^2}P_G\big)(v,v)\right]dv_M(v)+ O(p^{-\infty}).
\end{equation}

By Theorems \ref{Euler-et-chaleur} and \ref{thm-limit-near}, \eqref{trace=trace-sur-U}, and using the change of variable $Z^\perp\leftrightarrow \sqrt{p}Z^\perp$, we deduce that for every $u>0$,
\begin{multline}
\label{caractŽristique-et-integrale-1}
 p^{-n+d} \sum_{j=0}^q (-1)^{q-j}b_{j}^{p,G}   \leq\\
  \frac{\rank(E)}{(2\pi)^{n-d}} \int_{x\in M_G,\, |Z^\perp|\leq\sqrt{p} \e} \!\frac{\det(\dot{R}_x^{L_G})\sum_{j=0}^q(-1)^{q-j}\tr_{\Lambda^{0,j}}[e^{2u\omega_{d}(x)}]}{\det\big(1-\exp(-2u\dot{R}_x^{L_G})\big)} e^{-u\LL^\perp_x}(Z^\perp,Z^\perp) dv_{TB}(x,Z^\perp) \\
  +o(1).
\end{multline}

For $u>0$, set
\begin{equation}
f(u)= \frac{1}{\tanh(2u)}-\frac{1}{\sinh(2u)}.
\end{equation}
 Then there is $c>0$ such that for $u> 1$, $f(u)> c$, and $f(u)\limarrow{u}{\pm\infty}\pm1$.
By \eqref{def-Lperp} and Mehler's formula (see \cite[Thm. E.1.4]{ma-marinescu} for instance), we know that
\begin{equation}
\label{noyau-Lperp(Zperp,Zperp)}
e^{-u\LL^\perp_x}(Z^\perp,Z^\perp) = \prod_{i=1}^d \sqrt{\frac{a_i^\perp}{\pi(1-e^{-4ua_i^\perp})}}\exp\left\{ -a_i^\perp f(ua_i^\perp) Z_i^{\perp,2} \right\}
\end{equation}
Thus, as $a_i^\perp f(ua_i^\perp)>0$,
\begin{equation}
\label{integrale-sur-NG}
\begin{aligned}
 \int_{ |Z^\perp|\leq\sqrt{p} \e}e^{-u\LL^\perp_x}(Z^\perp,Z^\perp)dv_{N_{G,x}}(Z^\perp) &=  \int_{ \R^d}e^{-u\LL^\perp_x}(Z^\perp,Z^\perp)dv_{N_{G,x}}(Z^\perp)   +O(p^{-\infty}) \\
&=\prod_{i=1}^d \sqrt{\frac{1}{ f(ua_i^\perp)(1-e^{-4ua_i^\perp})}}+O(p^{-\infty}) .
\end{aligned}
\end{equation}

Let $\{w_j^0\}$ be a local orthonormal frame of $T^{(1,0)}M_G$ such that $\dot{R}^{L_G}w_j^0=a_j^0\,w_j^0$ (see \eqref{def-aj}). Its dual frame is denoted by $\{w^{0,j}\}$. Then
\begin{equation}
\omega_{G,d} = -\sum_{j=1}^{n-d} a_j^0\, \bw^{0,j}\wedge i_{\bw_j^0}.
\end{equation}
We again denote by $w_j^0$ the horizontal lift of $w_j^0$ in $T^HP$. In the same way, let $\{w_j^\perp\}$ be a local orthonormal frame of $(TY\oplus JTY)^{(1,0)}$ such that $\dot{R}^{L}w_j^\perp=a_j^\perp\, w_j^0$ (see Section \ref{Sect-calcul}). Its dual frame is denoted by $\{w^{\perp,j}\}$. Then
\begin{equation}
\omega_{d} = -\sum_{j=1}^{n-d} a_j^0\, \bw^{0,j}\wedge i_{\bw_j^0} -\sum_{j=1}^{d} a_j^\perp\, \bw^{\perp,j}\wedge i_{\bw_j^\perp}.
\end{equation}

Thus, writing $\{w_j\}=\{w_j^0,w_j^\perp\}$ and $\{a_j\}=\{a_j^0\,,a_j^\perp\}$, we get
\begin{equation}
e^{2u\omega_{d} } = 1+\sum_j(e^{-2ua_j}-1)\bw^{j}\wedge i_{\bw_j},
\end{equation}
and
\begin{equation}
\label{trace(exp(omegad))}
\tr_{\Lambda^{0,q}}[e^{2u\omega_{d}}]=\sum_{j_1<\dots<j_q}\exp\Big(-2u\sum_{k=1}^qa_{j_k}\Big).
\end{equation}
In particular, there exist $C>0$ such that for $x\in M_G$, $u>1$ and $0\leq q \leq n$,
\begin{equation}
\label{domination}
\left|\frac{\det(\dot{R}_x^{L_G})\tr_{\Lambda^{0,q}}[e^{2u\omega_{d}(x)}]}{\det\big(1-\exp(-2u\dot{R}_x^{L_G})\big)}\prod_{i=1}^d\sqrt{\frac{1}{ f(ua_i^\perp)(1-e^{-4ua_i^\perp})}}\right|\leq C.
\end{equation}

On the other hand the signature of $b^L$ on $JTY$ is the same as on $TY$ (i.e., $(r,d-r)$), so by Lemma \ref{bL-ND} and \eqref{courbure-sur-MG}, \eqref{decompo-TU|P} and \eqref{noyau-descente} we have for $0\leq q \leq n$
\begin{equation}
\label{M(q)-et-MG(q)}
\pi(P\cap M( q)) = M_G(q-r),
\end{equation}
where $M(\leq q)$ is define in an analogue way as $M_G(\leq q)$ in the introduction. Thus, by \eqref{trace(exp(omegad))} and \eqref{M(q)-et-MG(q)},
\begin{multline}
\label{lim-infini-partie-MG}
\lim_{u\to+\infty}\frac{\det(\dot{R}_x^{L_G})\tr_{\Lambda^{0,q}}[e^{2u\omega_{d}(x)}]}{\det\big(1-\exp(-2u\dot{R}_x^{L_G})\big)}\prod_{i=1}^d\sqrt{\frac{1}{ f(ua_i^\perp)(1-e^{-4ua_i^\perp})}}\\
=\mathbf{1}_{M_G(q-r)}(x)(-1)^{q-r}\det(\dot{R}^{L_G}),
\end{multline}
where the function $\mathbf{1}_{S}$ takes the value 1 on $S$ and 0 elsewhere.

Using \eqref{caractŽristique-et-integrale-1}, \eqref{integrale-sur-NG}, \eqref{domination}, \eqref{lim-infini-partie-MG} and dominated convergence as $u\to+\infty$, we find
\begin{equation}
\label{presque-main-thm}
\begin{aligned}
&\limsup_{p\to+\infty} p^{-n+d} \sum_{j=0}^q (-1)^{q-j}b_{j}^{p,G} \leq  \frac{\rank(E)}{(2\pi)^{n-d}}
 \prod_{i=1}^d \sqrt{\frac{1}{ f(ua_i^\perp)(1-e^{-4ua_i^\perp})}}\times \\
&\qquad\qquad\qquad\qquad\int_{M_G} \frac{\det(\dot{R}_x^{L_G})\sum_{j=0}^q(-1)^{q-j}\tr_{\Lambda^{0,j}}[e^{2u\omega_{G,d}(x)}]}{\det\big(1-\exp(-2u\dot{R}_x^{L_G})\big)}  dv_{M_G}(x)\\
&\quad\leq (-1)^{q-r}  \int_{M_G(\leq q-r)}\det\Big(\frac{\dot{R}^{L_G}_x}{2\pi}\Big)dv_{M_G}(x).
\end{aligned}
\end{equation}

Finally, note that 
\begin{equation}
\label{dvMG-et-omegaG}
\det\Big(\frac{\dot{R}^{L_G}_x}{2\pi}\Big)dv_{M_G}(x) = \Big(\frac{\ic}{2\pi}R^{L_G}\Big)^{n-d}/(n-d)! = \frac{\omega_G^{n-d}}{(n-d)!}.
\end{equation}
Then \eqref{presque-main-thm} and \eqref{dvMG-et-omegaG} entail Theorem \ref{IMI}.


\subsection{The case of a locally free action}
\label{Sect-action-loc-libre}

In this section, we prove Theorem \ref{IMI} under Assumption \ref{assum-regular-value}. In particular, the action of $G$ on $P$ and $\ol{U}$ is only locally free, and thus $M_G$ and $B$ are orbifolds. The proof relies on a similar method as the case of a free $G$-action, but the main difference is that we need to work off-diagonal to get uniform estimates near the orbifold singularities. We explain below how to adapt the arguments in Sections  \ref{Sect-AD-heat-kernel} and \ref{Sect-proof} to get the general result.

Recall that $G^0=\{g\in G \: :  \:  g\cdot x=x \text{ for any }x\in M\}$. Then $G^0$ is a finite normal subgroup of $G$ and the quotient $G/G^0$ acts effectively on $M$.

It is a well-known fact that if $\phi\colon (M,g^{TM}) \to (M,g^{TM})$ is an isometry and $x\in M$ is a point such that $\phi(x)=x$ and $d\phi_x =\Id_{T_xM}$ then $\phi=\Id_M$. In particular, suppose that  $g\in G$ satisfies $g|_P=\Id_P$. Then we have for $x\in P$: $gx=x$, $dg_x|_{T_xP}=\Id_{T_xP}$ and $g$ preserves $J$ so $dg_x|_{JT_xP}=\Id_{JT_xP}$. As $TP+JTP=TM$, we deduce that $g$ acts as the identity on $M$. Thus,
\begin{equation}
\label{G0(M)=G0(P)}
G^0=\{g\in G \: :  \:  g\cdot x=x \text{ for any }x\in P\}.
\end{equation}

Recall that  the function $h$ defined in \eqref{def-h} is smooth only on the regular part of $B$ and we have denoted by $\wh{h}$ its smooth extension from the regular part of $B$ to $B$.

First, we need to modify Section \ref{Sect-rescaling} as follows.

Recall that $TM$ is endowed with a metric $g^{TM}$ satisfying \eqref{g^TM|_P}. We identify the normal bundle $N$ of $P$ in $U$ to the orthogonal complement of $TP$. By \eqref{TP=TY+THP} and \eqref{g^TM|_P}, this means that $N$ is identified with $JTY$. By  \eqref{TP=TY+THP} and \eqref{T^HU=JTP-eq}, we have in particular $ T^HU=T^HP\oplus N$.

Let $g^{TY}$, $g^{T^HU}$  be the restriction of $g^{TM}$ on $TY$, $T^HU$. Let $g^{TB}$ (resp. $g^{TM_G}$) be the metric on $TB$ (resp. $TM_G$) induced by $g^{T^HU}$ (resp. $g^{T^HP}$).

Here, unlike in Section \ref{Sect-AD-heat-kernel}, we will not work on the quotient $B$ but directly on $M$. Let  $\n^{TB}$ be the Levi-Civita connection on $(TB,g^{TB})$. Let $P^{N}$ and $P^{T^HP}$ be the orthogonal projections from $T^HU|_{P}$ to $N$ and $T^HP$ respectively. Set
\begin{equation}
\begin{aligned}
& \n^{T^HU} = \pi^*\n^{TB}, && \n^{N} = P^{N}(\n^{T^HU}|_{P})P^{N}, \\
 &\n^{T^HP} = P^{T^HP}(\n^{T^HU}|_{P})P^{T^HP}, && {}^0\n^{T^HU} = \n^{N}\oplus \n^{T^HP}.
\end{aligned}
\end{equation}

Fix $y_0 \in P$. For $V\in T^HU$ (resp. $T^HP$), we define $t \mapsto x_t=\exp_{y_0}^{T^HU}(tV)\in U$ (resp. $\exp_{y_0}^{T^HP}(tV)\in P$) the curve such that $x_0 = y_0$, $\dot{x}_0=V$, $\dot{x} \in T^HU$ and $\n^{T^HU}_{\dot{x}}\dot{x}=0$ (resp. $\dot{x} \in T^HP$ and $\n^{T^HP}_{\dot{x}}\dot{x}=0$). For $W\in T^HP$ small and $V\in N_{y_0}$, let $\tau_W V$ be the parallel transport of $V$ with respect to $\n^{N}$ along to curve $t\in [0,1] \mapsto  \exp_{y_0}^{T^HP}(tW)$.

As in Section \ref{Sect-rescaling}, we identify $B^{T_{y_0}^HU}(0,\e)$ to a subset of $U$ as follows: for $Z\in B^{T_{y_0}^HU}(0,\e)$, we decompose $Z$ as $Z=Z^0+Z^\perp$ with $Z^0\in T_{y_0}^HP$ and $Z^\perp\in N_{y_0}$, and then we identify $Z$ with $\exp^{T^HU}_{\exp^{T^HP}_{y_0}(Z^0)}(\tau_{Z^0} Z^\perp)$. 

Moreover, if  $G_{y_0} = \{g\in G \: : \: gy_0=y_0\}$ is the stabilizer of $y_0$ and $g\in G_{y_0}$, we can decompose $T_{y_0}^HP$ as 
\begin{equation}
T_{y_0}^HP=(T_{y_0}^HP)^g\oplus \mathcal{N}_{y_0,g} ,
\end{equation}
where $(T_{y_0}^HP)^g$ is the fixed point-set of $g$ in $T_{y_0}^HP=T_{y_0}P\cap JT_{y_0}P$ and $\mathcal{N}_{y_0,g}$ is its orthogonal complement. Hence we get, for each $g\in G_{y_0}$, a decomposition of the coordinate $Z^0$ as $Z^0= Z^0_{1,g}+Z^0_{2,g}$ with $Z^0_{1,g}\in (T_{y_0}^HP)^g$ and $Z^0_{2,g}\in \mathcal{N}_{y_0,g}$. Note that $\mathrm{rk}(\mathcal{N}_{y_0,g})=0$ if and only if $g\in G^0$.

Observe that $U\simeq G\cdot B^{T_{y_0}^HU}(0,\e) =G\times_{G_{y_0}}B^{T_{y_0}^HU}(0,\e)$ is a $G$-neighborhood of the orbit $G\cdot y_0$ and $(B^{T_{y_0}^HU}(0,\e),G_{y_0})$ gives local chart on $B$.

As the constructions in Section \ref{Sect-rescaling} are $G_{y_0}$-equivariant, we can extend in the same way the geometric objects from $G\times_{G_{y_0}}B^{T_{y_0}^HU}(0,\e)$ to 
\begin{equation}
M_0:= G\times_{G_{y_0}}\R^{2n-d},
\end{equation}
 where $\R^{2n-d}\simeq T_{y_0}^HU$. Note that Lemma \ref{noyau-Lp-et-LpM0} and Corollary \ref{noyau-Lp-et-LpM0-sur-invariant} still hold, because do not work on the quotient to get them: we only use finite propagation speed of the wave equation on $M$.

Set 
\begin{equation}
\begin{aligned}
&B_0=M_0/G=\R^{2n-d}/G_{y_0}, \\
&\wh{M}_0 = G\times \R^{2n-d}, & \wh{B}_0=\wh{M}_0/G=\R^{2n-d}.
\end{aligned}
\end{equation}
Then we have a covering $\wh{M}_0\to M_0$ (resp. $\wh{B}_0\to B_0$) which gives a (global) orbifold chart on $M_0$ (resp. $B_0$). We can then extend the geometric objects from $M_0$ to $\wh{M}_0$. We will add a hat to denote the corresponding objects on $\wh{B}_0$ or $\wh{M}_0$.  In particular, we have  a Dirac operator $D^{\wh{M}_0}_p$ on $\wh{M}_0$ corresponding to $D^{M_0}_p$ in Section \ref{Sect-rescaling}.

Let $\wh{\pi}_G \colon G\times \R^{2n-d} \to \R^{2n-d}$ be the projection on the second factor. As in \eqref{def-Phi}, we define
\begin{equation}
\wh{\Phi} = \wh{h}\wh{\pi}_G \colon \smooth(G\times \R^{2n-d}, \E_{0,p})^G\to \smooth( \R^{2n-d}, (\E_{0,p})_{\wh{B}_0}).
\end{equation}
We also denote by $\wh{\Phi}$ the map induced from $\smooth(M_0, \E_{0,p})^G$ to $\smooth( B_0, (\E_{0,p})_{B_0})$.

Let $g^{TM_0}$ be defined as in \eqref{gTM0-et-J0} and let $g^{T^HM_0}$ be the metric on $\R^{2n-d}$ induced by $g^{TM_0}$, with corresponding Riemannian volume on $(\R^{2n-d}, g^{T^HM_0})$ denoted by $dv_{T^HM_0}$.

Let $e^{-u\wh{\Phi}D^{\wh{M}_0,2}_p\wh{\Phi}}$ be the heat kernel of the operator $\wh{\Phi}D^{\wh{M}_0,2}_p\wh{\Phi}$ on  $\wh{B}_0$ and $e^{-u\wh{\Phi}D^{\wh{M}_0,2}_p\wh{\Phi}}(Z,Z')$ ($Z,Z'\in \wh{B}_0$) be its smooth kernel with respect to $dv_{T^HM_0}(Z')$. Concerning heat kernels on orbifolds, we refer the reader to \cite[Sect. 2.1]{MaTAMS}. Then we have for $v=[g,Z]$ and $v'=[g',Z']$ in $M_0$,
\begin{equation}
\begin{aligned}
\label{noyau-et-Phi-chapeau}
\wh{h}(v)\wh{h}(v')\big(P_Ge^{-\frac{u}{p}D_p^{M_0,2}}P_G\big)(v,v') &= e^{-\frac{u}{p}\wh{\Phi} D_p^{M_0,2}\wh{\Phi}^{-1}}(\pi(v),\pi(v')) \\
&= \frac{1}{|G^0|}\sum_{g\in G_{y_0}} (g,1)\cdot e^{-\frac{u}{p}\wh{\Phi}D^{\wh{M}_0,2}_p\wh{\Phi}}(g^{-1}Z,Z'),
\end{aligned}
\end{equation}
where $|G^0|$ is the cardinal of $G^0$. Indeed, the first equality in \eqref{noyau-et-Phi-chapeau} is analogous to \eqref{noyau-et-Phi}, and the second from a similar computation as in \cite[(5.19)]{MR2215454} or \cite[(5.4.17)]{ma-marinescu}.

Note that our trivialization of the restriction of $L$ (resp. $E$) on $B^{T_{y_0}^HU}(0,\e)$ is not $G_{y_0}$-invariant, except if $G_{y_0}$ acts trivially on $L_{y_0}$ (resp. $E_{y_0}$). More precisely, let $\wh{M}_{G,0} = \R^{2n-2d}\times \{0\}\subset \wh{B}_0$ and for $g\in G_{y_0}$, let $\wh{M}_{G,0}^g$ be the fixed point-set of $g$ in $\wh{M}_{G,0}$. Then the action of $g$ on $L|_{\wh{M}_{G,0}^g}$ is the multiplication by $e^{i\theta_g}$ and $\theta_g$ is locally constant on $\wh{M}_{G,0}^g$. Likewise, the action of $g$ on $E|_{\wh{M}_{G,0}^g}$ is given by $g_E\in \smooth(\wh{M}_{G,0}^g,\End(E))$ which is parallel with respect to $\n^E$.

Now, as we work on $\wh{B}_0$ and $\wh{M}_0$, we can apply the results of Sections \ref{Sect-rescaling}-\ref{Sect-calcul} to the operator $\wh{\Phi}D^{\wh{M}_0,2}_p\wh{\Phi}$. We will use the same notation as in these sections, and add a subscript to indicate the base-point (e.g., $\kappa_x$, $\LL_{0,x}$, ...). By Theorem \ref{cvcenoyauLup->noyauLuinfini} and \eqref{noyauxdeMupetdeLup}, we obtain for $g\in G_{y_0}$ and $u>0$ fixed
\begin{multline}
\label{chaleur-Lp-et-L0-chapeau-1}
 \Big| p^{-n+d/2}e^{-\frac{u}{p}\wh{\Phi}D^{\wh{M}_0,2}_p\wh{\Phi}}(g^{-1}Z,Z) \\
 - \kappa_{Z_{1,g}}^{-1}(Z^\perp)e^{-u\LL_{0,Z_{1,g}}}\big(\sqrt{p}g^{-1}(Z_{2,g}+Z^\perp),\sqrt{p}(Z_{2,g}+Z^\perp)\big) \Big|_{\mathscr{C}^{m'}(M_G)} \\
\leq Cp^{-1/2}  \big(1+\sqrt{p}|Z_{2,g}|\big)^N\big(1+\sqrt{p}|Z^\perp|\big)^{-m}\exp \big(-Cp\inf_{h\in G_{y_0}}|h^{-1}Z-Z|^2\big).
\end{multline}
On the other hand, note that there is $\rho>0$ such that  for $g\in G_{y_0}$, $|g^{-1}Z-Z|^2 \geq \rho |Z_{2,g}|^2$, so 
\begin{multline}
\label{chaleur-Lp-et-L0-chapeau-2}
 \Big| p^{-n+d/2}e^{-\frac{u}{p}\wh{\Phi}D^{\wh{M}_0,2}_p\wh{\Phi}}(g^{-1}Z,Z) \\
 - \kappa_{Z_{1,g}}^{-1}(Z^\perp)e^{-u\LL_{0,Z_{1,g}}}\big(\sqrt{p}g^{-1}(Z_{2,g}+Z^\perp),\sqrt{p}(Z_{2,g}+Z^\perp)\big) \Big|_{\mathscr{C}^{m'}(M_G)} \\
\leq Cp^{-1/2} \big(1+\sqrt{p}|Z^\perp|\big)^{-m}\exp \big(-C'p|Z_{2,g}|^2\big).
\end{multline}

We can now prove Theorem \ref{IMI-general}. First, observe that Theorem  \ref{Euler-et-chaleur} is still true here because we work on $M$ to prove it in Section \ref{Sect-proof}. Thus, we can use a similar approach to prove Theorem~\ref{IMI-general} as in Section \ref{Sect-proof}.

Note that the estimate \eqref{trace=trace-sur-U} still holds. Consider now a $G$-invariant function $\psi \in \smooth(M)$ such that the induced  function (again denoted by $\psi$) on $B$ is compactly supported in a small neighborhood of $x_0\in M_G$. 

Similarly to \eqref{def-kappa}, we denote by $dv_{T^HU}$ the Riemannian volume of $(T^H_{y_0}U, g^{T^H_{y_0}U})$. Then, as in \eqref{caractŽristique-et-integrale-1}, \eqref{chaleur-Lp-et-L0-chapeau-2} and dominated convergence imply that
\begin{multline}
\label{integral-avec-tous-les-g}
 p^{-n+d} \int_{U} \psi(v)\tr_q\left[ \big(P_Ge^{-\frac{u}{p}D_p^{M_0,2}}P_G\big)(v,v)\right] dv_M(v) =\\
  \frac{1}{|G_{y_0}/G^0||G^0|}\sum_{g\in G_{y_0}} \frac{p^{-\mathrm{rk}(N_{y_0,g})/2}}{(2\pi)^{n-d}} \int_{A(p,\e)} \psi\Big(Z_{1,g}+\frac{Z_{2,g}}{\sqrt{p}}\Big)\tr_q\Bigg[(g,1)\cdot \frac{\det(\dot{R}_{Z_{1,g}}^{L_G})e^{2u\omega_{d}(Z_{1,g})}}{\det\big(1-\exp(-2u\dot{R}_{Z_{1,g}}^{L_G})\big)} \\
   \times e^{-u\LL^\perp_{Z_{1,g}}}\big(g^{-1}(Z_{2,g}+Z^\perp),Z_{2,g}+Z^\perp\big)\otimes \Id_E \Bigg]dv_{T^HU}(Z) 
  +o(1),
\end{multline}
where $A(p,\e)= \{Z\in \wh{B}_0\,:\,|Z_{1,g}|\leq \e,|Z_{2,g}|\leq \e\sqrt{p},|Z^\perp|\leq \e\sqrt{p} \}$. In particular, in \eqref{integral-avec-tous-les-g}, every term involving a $g$ such that $\mathrm{rk}(N_{y_0,g})>0$, i.e., $g\notin G^0$, disappears when we look at the leading term in $p$.

Thus, we now consider $g\in G^0$. The action of $g$ on $M$ and $\Wedge(T^*M)$ is trivial, so we have
\begin{multline}
\label{trace-Omega(0,q)-et-action-de-g}
\tr_q\left[(g,1)\cdot \frac{\det(\dot{R}_{Z^0}^{L_G})e^{2u\omega_{d}(Z^0)}}{\det\big(1-\exp(-2u\dot{R}_{Z^0}^{L_G})\big)} 
   \times e^{-u\LL^\perp_{Z^0}}(g^{-1}Z^\perp,Z^\perp) \otimes \Id_E\right]  \\
 = e^{ip\theta_g}\frac{\det(\dot{R}_{Z^0}^{L_G})\tr_{\Lambda^{0,q}}[e^{2u\omega_{d}(Z^0)}]}{\det\big(1-\exp(-2u\dot{R}_{Z^0}^{L_G})\big)}e^{-u\LL^\perp_{Z^0}}(Z^\perp,Z^\perp)\otimes g_E(Z^0).
\end{multline}

Using \eqref{integral-avec-tous-les-g} and \eqref{trace-Omega(0,q)-et-action-de-g}, we get as in \eqref{integrale-sur-NG}-\eqref{presque-main-thm}:

\begin{align}
\label{presque-main-thm-orbifold}
&\limsup_{p\to+\infty} p^{-n+d} \int_{U} \psi(v) \tr_q\left[\big(P_Ge^{-\frac{u}{p}D_p^2}P_G\big)(v,v)\right]dv_M(v) \notag \\
&\qquad\qquad\qquad\qquad \leq  \frac{1}{(2\pi)^{n-d}} \frac{1}{|G^0|}\sum_{g\in G^0}\prod_{i=1}^d \sqrt{\frac{1}{ f(ua_i^\perp)(1-e^{-4ua_i^\perp})}}\times \notag \\
&\qquad\qquad\qquad\qquad\qquad  \int_{M_G}\psi(x) \frac{\det(\dot{R}_x^{L_G})\sum_{j=0}^q(-1)^{q-j}\tr_{\Lambda^{0,j}}[e^{2u\omega_{G,d}(x)}]}{\det\big(1-\exp(-2u\dot{R}_x^{L_G})\big)}  e^{ip\theta_g}\tr^E[g_E(x)]dv_{M_G}(x)  \\
&\qquad\qquad\qquad\qquad \leq (-1)^{q-r}  \int_{M_G(\leq q-r)}\psi(x)\det\Big(\frac{\dot{R}^{L_G}_x}{2\pi}\Big)\frac{1}{|G^0|}\Big(\sum_{g\in G^0}e^{ip\theta_g}\tr^E[g_E(x)]\Big)dv_{M_G}(x) \notag \\
&\qquad\qquad\qquad\qquad\qquad = (-1)^{q-r} \dim(L^p\otimes E)^{G^0} \int_{M_G(\leq q-r)}\psi(x)\det\Big(\frac{\dot{R}^{L_G}_x}{2\pi}\Big)dv_{M_G}(x).\notag 
\end{align}

Finally, we take some functions $\psi_k$ as $\psi$ above and such that $\sum_k \psi_k=1$ in a neighborhood of $M_G$ in $B$ and we apply \eqref{presque-main-thm-orbifold} for those $\psi_k$. We get Theorem \ref{IMI-general} by taking the sum over $k$ of the obtained estimates and using Theorem \ref{Euler-et-chaleur} and \eqref{trace=trace-sur-U}. 

\subsection{The other isotypic components of the cohomology}
\label{autres-composantes}

In this subsection, we show how to use  Theorem \ref{IMI-general} to get estimates on the other isotypic components of the cohomology $H^\bullet(M,L^p\otimes E)$. 

Let $\mathcal{V}_{\gamma}$ be the finite dimensional irreducible representation of  $G$ with highest weight $\gamma$.

For a representation $F$ of $G$, we denote by $F_{\gamma}$ its isotopic component associated with $\gamma$. Then  we have
\begin{equation}
\label{compo-2pigamma}
\begin{aligned}
H^\bullet(M,L^p\otimes E)_{\gamma} &= \V\otimes \Hom_G(\V,H^\bullet(M,L^p\otimes E))\\
&=\V\otimes\big(H^\bullet(M,L^p\otimes E)\otimes \V^*\big)^G\\
&= \V \otimes H^\bullet(M,L^p\otimes E \otimes \V^*)^G,
\end{aligned}
\end{equation}
where $\V^*$ is viewed as a trivial bundle over $M$.

By Theorem \ref{IMI-general} applied replacing $E$ by $E\otimes \V^*$ and \eqref{compo-2pigamma} we have as $p\to +\infty$, 
 \begin{multline}
\sum_{j=0}^q (-1)^{q-j}\dim H^j(M,L^p\otimes E)_{\gamma}  \\
\leq \dim \V \dim(L^p\otimes E\otimes \V^*)^{G^0} \frac{p^{n-d}}{(n-d)!}  \int_{M_G(\leq q-r)} (-1)^{q-r} \omega_G^{n-d} + o(p^{n-d}),
\end{multline}
with equality for $q=n$.

In particular, we get the weak inequalities
 \begin{multline}
\dim H^q(M,L^p\otimes E)_{\gamma}  \\
\leq \dim \V \dim(L^p\otimes E\otimes \V^*)^{G^0} \frac{p^{n-d}}{(n-d)!}  \int_{M_G( q-r)} (-1)^{q-r} \omega_G^{n-d} + o(p^{n-d}).
\end{multline}


\end{document}